\documentclass[12pt]{amsart}
\usepackage{etex}
\usepackage{amssymb}
\usepackage{amsfonts}
\usepackage{amsmath}
\usepackage{array}
\usepackage{rotating}

\usepackage{pstricks}

\usepackage{graphicx}
\usepackage{placeins}
\usepackage{pspicture}

\usepackage{tikz}

\usepackage[matrix,ps,xdvi]{xypic} 
\newdimen\vcadre\vcadre=0.1cm 
\newdimen\hcadre\hcadre=0.1cm 
\def\GrTeXBox#1{\vbox{\vskip\vcadre\hbox{\hskip\hcadre%
      $#1$%
   \hskip\hcadre}\vskip\vcadre}}
\def\arx#1[#2]{\ifcase#1 \relax \or%
  \ar @{-}[#2]  \or%
  \ar @2{-}[#2] \or%
  \ar @{--}[#2] \or%
  \ar @2{.}[#2] \or%
  \ar @{~}[#2]  \fi}

\def\arbgx#1#2{
\newdimen\vcadre\vcadre=0.01cm 
\newdimen\hcadre\hcadre=0.01cm 
\xymatrix@R=0.1cm@C=1mm{
 & {\GrTeXBox{\bullet}}\arx1[dl]\arx1[dr]\\
 {\GrTeXBox{#1}} & *{} & {\GrTeXBox{#2}} \\
}
}

\def\arbgxb#1#2#3{
\newdimen\vcadre\vcadre=0.01cm 
\newdimen\hcadre\hcadre=0.01cm 
\xymatrix@R=0.1cm@C=2mm{
 & {\GrTeXBox{\bullet}}\arx1[dl]\arx1[dr]\\
 {\GrTeXBox{#1}} & *{} & {\GrTeXBox{\bullet}}\arx1[dl]\arx1[dr]\\
 & {\GrTeXBox{#2}} & *{} & {\GrTeXBox{#3}} & *{} \\
}
}

\def\arbgxc#1#2#3{
\newdimen\vcadre\vcadre=0.01cm 
\newdimen\hcadre\hcadre=0.01cm 
\xymatrix@R=0.1cm@C=2mm{
 && {\GrTeXBox{\bullet}}\arx1[dl]\arx1[dr]\\
 & {\GrTeXBox{\bullet}}\arx1[dl]\arx1[dr]  & *{} & {\GrTeXBox{#3}} \\
 {\GrTeXBox{#1}} & *{} & {\GrTeXBox{#2}} & *{} \\
}
}

\newtheorem{example}{Example}[section]

\newtheorem{note}[example]{Note}
\newtheorem{theorem}[example]{Theorem}

\newtheorem{corollary}[example]{Corollary}
\newtheorem{conjecture}[example]{Conjecture}

\newtheorem{proposition}[example]{Proposition}

\newtheorem{lemma}[example]{Lemma}

\def\Proof{\noindent \it Proof -- \rm}
\def\qed{\hspace{3.5mm} \hfill \vbox{\hrule height 3pt depth 2 pt width 2mm}
\bigskip}

\def\gf#1#2{\genfrac{}{}{0pt}{}{#1}{#2}}
\newcommand{\tdelta}{\tilde{\Delta}}

\def\gr{{\rm gr}}

\def\std{{\rm std}}
\def\des{{\rm des}}
\def\maj{{\rm maj}}

\def\FQSym{{\bf FQSym}}

\def\Lie{{\rm Lie}}
\def\<{\langle}
\def\>{\rangle}
\def\ca{{\bf ca}}

\def\U{{\sf U}}

\def\CC{{\mathfrak C}}

\def\F{{\bf F}}

\def\G{{\bf G}}

\def\P{{\bf P}}

\def\SG{{\mathfrak S}}
\def\H{{\bf H}}

\def\K{{\mathbb K}}

\def\Sym{{\bf Sym}}
\def\NCSF{{\bf Sym}}

\def\Des{\operatorname{Des}}

\def\dim{{\rm dim}}

\def\End{\operatorname{End}}
\def\DT{{\rm DT}}
\def\PL{{\rm PL}}
\def\can{{\rm can}}

\def\LC{{\mathcal B}}
\def\SC{{\mathcal L}}

\def\PBT{{\bf PBT}}
\def\shuff#1#2{\mathbin{
\hbox{\vbox{ \hbox{\vrule \hskip#2 \vrule height#1 width 0pt
}%
\hrule}%
\vbox{ \hbox{\vrule \hskip#2 \vrule height#1 width 0pt
\vrule }%
\hrule}%
}}}

\def\shuf{{\mathchoice{\shuff{7pt}{3.5pt}}%
{\shuff{6pt}{3pt}}%
{\shuff{4pt}{2pt}}%
{\shuff{3pt}{1.5pt}}}}%
\def\shuffle{\,\shuf\,}

\def\gaudend{\!\prec\!}   
\def\droitdend{\!\succ\!} 

\def\droit{\triangleright}

\def\ZZ{{\mathbb Z}}    

\def\T{{\sf T}}
\def\c{{\sf c}}
\def\x{{\sf x}}

\def\nesw{{\nearrow\swarrow}}

\def\btr{{\blacktriangleright}}

\newdimen\Squaresize \Squaresize=14pt
\newdimen\Thickness \Thickness=0.5pt

\def\Square#1{\hbox{\vrule width \Thickness
   \vbox to \Squaresize{\hrule height \Thickness\vss
      \hbox to \Squaresize{\hss#1\hss}
   \vss\hrule height\Thickness}
\unskip\vrule width \Thickness}
\kern-\Thickness}

\def\Vsquare#1{\vbox{\Square{$#1$}}\kern-\Thickness}

\def\Tabvrule{\vrule width-0.4pt}       
\def\Tabhrule{\hrule \hrule height-0.4pt} 
\def\Tabstrut{\vrule height2.2ex 
                     depth0.8ex  
                     width0ex    
\relax}

\def\PasCase#1{\omit%
            $\vcenter{\hbox {\vbox to 0.4pt{}}
               \hbox{\makebox[3ex]{\Tabstrut$#1$}}}%
               \Tabvrule$}
\def\PasCasePoint{\PasCase{\cdot}}
\def\DessinCarre#1{%
    \vcenter{\hbox{}\hrule
             \hbox{\vrule\makebox[3ex]{\Tabstrut$#1$}\vrule}\Tabhrule}%
             \Tabvrule}
\def\GenRuban#1{\vcenter{\halign{&$\DessinCarre{##}$\cr#1}}\egroup}

\def\sTabvrule{\vrule width-0.4pt}
\def\sTabhrule{\hrule \hrule height-0.4pt}
\def\sTabstrut{\vrule height1.6ex depth0.6ex width0ex \relax}
\def\sDessinCarre#1{%
    \vcenter{\hbox{}\hrule
             \hbox{\vrule\makebox[2.3ex]%
                  {\sTabstrut$\scriptstyle#1$}\vrule}\sTabhrule}%
             \sTabvrule}
\def\sGenRuban#1{\vcenter{\halign{&$\sDessinCarre{##}$\cr#1}}\egroup}

\def\ruban{%
  \bgroup
  \let\ =\omit
  \let\\=\cr
  \let\x=\times
  \let\.=\PasCasePoint
  \offinterlineskip
  \GenRuban}

\def\sruban{%
  \bgroup
  \let\ =\omit
  \let\x=\times
  \let\\=\cr
  \offinterlineskip
  \sGenRuban}

\def\Btabc{\begin{picture}(3,4)\put(1,1){\circle*{0.7}}\put(2,2){\circle*{0.7}}\put(2,2){\Line(-1,-1)}\put(3,3){\circle*{0.7}}\put(3,3){\Line(-1,-1)}\put(3,3){\circle*{1}}\end{picture} }

\def\Btbac{ \begin{picture}(3,4)\put(1,2){\circle*{0.7}}\put(2,1){\circle*{0.7}}\put(1,2){\Line(1,-1)}\put(3,3){\circle*{0.7}}\put(3,3){\Line(-2,-1)}\put(3,3){\circle*{1}}\end{picture} }

\def\Btacb{ \begin{picture}(3,3)\put(1,1){\circle*{0.7}}\put(2,2){\circle*{0.7}}\put(3,1){\circle*{0.7}}\put(2,2){\Line(1,-1)}\put(2,2){\Line(-1,-1)}\put(2,2){\circle*{1}}\end{picture} }

\def\Btcab{ \begin{picture}(3,4)\put(2,1){\circle*{0.7}}\put(3,2){\circle*{0.7}}\put(2,1){\Line(1,1)}\put(1,3){\circle*{0.7}}\put(1,3){\Line(2,-1)}\put(1,3){\circle*{1}}\end{picture} }

\def\Btcba{ \begin{picture}(3,4)\put(1,3){\circle*{0.7}}\put(2,2){\circle*{0.7}}\put(3,1){\circle*{0.7}}\put(2,2){\Line(1,-1)}\put(1,3){\Line(1,-1)}\put(1,3){\circle*{1}}\end{picture} }

\def\arbuga{\begin{picture}(2,4)\cerg{1}3\cerp{1}1
 \put(1,3){\Line( 0,-2)}
\end{picture}}

\def\arbdga{\begin{picture}(2,6)\cerg15\cerp13\cerp11
 \put(1,5){\Line( 0,-2)}
 \put(1,3){\Line( 0,-2)}
\end{picture}}

\def\arbdgb{\begin{picture}(2,4)\cerg15\cerp13\cerp11
 \put(1,5){\Line( 0,-2)}
 \put(1,3){\Line( 0,-2)}
\end{picture}}

\def\arbdgb{\begin{picture}(3,4)\cerg23\cerp11\cerp31
 \put(2,3){\Line( 1,-2)}
 \put(2,3){\Line(-1,-2)}
\end{picture}}

\def\arbta{\begin{picture}(3,4)\put(1,1){\circle*{0.7}}\put(2,2){\circle*{0.7}}\put(2,2){\Line(-1,-1)}\put(3,3){\circle*{0.7}}\put(3,3){\Line(-1,-1)}\put(3,3){\circle*{1}}\end{picture}}

\def\arbtb{\begin{picture}(3,4)\put(1,2){\circle*{0.7}}\put(2,1){\circle*{0.7}}\put(1,2){\Line(1,-1)}\put(3,3){\circle*{0.7}}\put(3,3){\Line(-2,-1)}\put(3,3){\circle*{1}}\end{picture}}

\def\arbtc{\begin{picture}(3,4)\put(1,3){\circle*{0.7}}\put(2,1){\circle*{0.7}}\put(3,2){\circle*{0.7}}\put(3,2){\Line(-1,-1)}\put(1,3){\Line(2,-1)}\put(1,3){\circle*{1}}\end{picture}}

\def\arbtd{\begin{picture}(3,3)\put(1,1){\circle*{0.7}}\put(2,2){\circle*{0.7}}\put(3,1){\circle*{0.7}}\put(2,2){\Line(-1,-1)}\put(2,2){\Line(1,-1)}\put(2,2){\circle*{1}}\end{picture}}

\def\arbte{\begin{picture}(3,4)\put(1,3){\circle*{0.7}}\put(2,2){\circle*{0.7}}\put(3,1){\circle*{0.7}}\put(2,2){\Line(1,-1)}\put(1,3){\Line(1,-1)}\put(1,3){\circle*{1}}\end{picture}}

\newdimen\vcadre\vcadre=0.2cm 
\newdimen\hcadre\hcadre=0.2cm 

\def\cerp#1#2{\put(#1,#2){\circle*{0.7}}}
\def\cerg#1#2{\put(#1,#2){\circle*{1}}}

\def\arbtga{\begin{picture}(3,4)\cerg23\cerp{.5}1\cerp21\cerp{3.5}1
 \put(2,3){\Line(-1.5,-2)}
 \put(2,3){\Line( 0,-2)}
 \put(2,3){\Line( 1.5,-2)}
\end{picture}}

\def\arbtgb{\begin{picture}(3,6)\cerg25\cerp13\cerp33\cerp11
 \put(2,5){\Line(-1,-2)}
 \put(2,5){\Line( 1,-2)}
 \put(1,3){\Line( 0,-2)}
\end{picture}}

\def\arbtgd{\begin{picture}(3,6)\cerg25\cerp13\cerp33\cerp31
 \put(2,5){\Line(-1,-2)}
 \put(2,5){\Line( 1,-2)}
 \put(3,3){\Line( 0,-2)}
\end{picture}}

\def\arbtgc{\begin{picture}(3,6)\cerg25\cerp23\cerp11\cerp31
 \put(2,5){\Line( 0,-2)}
 \put(2,3){\Line( 1,-2)}
 \put(2,3){\Line(-1,-2)}
\end{picture}}

\def\arbtge{\begin{picture}(3,8)\cerg27\cerp25\cerp23\cerp21
 \put(2,7){\Line( 0,-2)}
 \put(2,5){\Line( 0,-2)}
 \put(2,3){\Line( 0,-2)}
\end{picture}}

\def\cerp#1#2{\put(#1,#2){\circle*{0.7}}}
\def\cerg#1#2{\put(#1,#2){\circle*{1}}}

\def\arbz{\begin{picture}(4,5)\cerp11\cerp22\put(2,2){\Line(-1,-1)}\end{picture}}

\def\arba{\begin{picture}(4,5)\cerp11\cerp22\cerp33\cerg44\put(2,2){\Line(-1,-1)}\put(3,3){\Line(-1,-1)}\put(4,4){\Line(-1,-1)}\end{picture}}

\def\arbb{\begin{picture}(4,5)\put(1,2){\circle*{0.7}}\put(2,1){\circle*{0.7}}\put(1,2){\Line(1,-1)}\put(3,3){\circle*{0.7}}\put(3,3){\Line(-2,-1)}\put(4,4){\circle*{0.7}}\put(4,4){\Line(-1,-1)}\put(4,4){\circle*{1}}\end{picture}}

\def\arbc{\begin{picture}(4,5)\put(1,3){\circle*{0.7}}\put(2,1){\circle*{0.7}}\put(3,2){\circle*{0.7}}\put(3,2){\Line(-1,-1)}\put(1,3){\Line(2,-1)}\put(4,4){\circle*{0.7}}\put(4,4){\Line(-3,-1)}\put(4,4){\circle*{1}}\end{picture}}

\def\arbd{\begin{picture}(4,5)\put(1,2){\circle*{0.7}}\put(2,3){\circle*{0.7}}\put(3,2){\circle*{0.7}}\put(2,3){\Line(-1,-1)}\put(2,3){\Line(1,-1)}\put(4,4){\circle*{0.7}}\put(4,4){\Line(-2,-1)}\put(4,4){\circle*{1}}\end{picture}}

\def\arbe{\begin{picture}(4,5)\put(1,4){\circle*{0.7}}\put(2,1){\circle*{0.7}}\put(3,2){\circle*{0.7}}\put(3,2){\Line(-1,-1)}\put(4,3){\circle*{0.7}}\put(4,3){\Line(-1,-1)}\put(1,4){\Line(3,-1)}\put(1,4){\circle*{1}}\end{picture}}

\def\arbf{\begin{picture}(4,5)\put(1,3){\circle*{0.7}}\put(2,2){\circle*{0.7}}\put(3,1){\circle*{0.7}}\put(2,2){\Line(1,-1)}\put(1,3){\Line(1,-1)}\put(4,4){\circle*{0.7}}\put(4,4){\Line(-3,-1)}\put(4,4){\circle*{1}}\end{picture}}

\def\arbg{\begin{picture}(4,4)\put(1,1){\circle*{0.7}}\put(2,2){\circle*{0.7}}\put(2,2){\Line(-1,-1)}\put(3,3){\circle*{0.7}}\put(4,2){\circle*{0.7}}\put(3,3){\Line(-1,-1)}\put(3,3){\Line(1,-1)}\put(3,3){\circle*{1}}\end{picture}}

\def\arbh{\begin{picture}(4,5)\put(1,4){\circle*{0.7}}\put(2,2){\circle*{0.7}}\put(3,1){\circle*{0.7}}\put(2,2){\Line(1,-1)}\put(4,3){\circle*{0.7}}\put(4,3){\Line(-2,-1)}\put(1,4){\Line(3,-1)}\put(1,4){\circle*{1}}\end{picture}}

\def\arbi{\begin{picture}(4,4)\put(1,2){\circle*{0.7}}\put(2,1){\circle*{0.7}}\put(1,2){\Line(1,-1)}\put(3,3){\circle*{0.7}}\put(4,2){\circle*{0.7}}\put(3,3){\Line(-2,-1)}\put(3,3){\Line(1,-1)}\put(3,3){\circle*{1}}\end{picture}}

\def\arbj{\begin{picture}(4,4)\put(1,2){\circle*{0.7}}\put(2,3){\circle*{0.7}}\put(3,1){\circle*{0.7}}\put(4,2){\circle*{0.7}}\put(4,2){\Line(-1,-1)}\put(2,3){\Line(-1,-1)}\put(2,3){\Line(2,-1)}\put(2,3){\circle*{1}}\end{picture}}

\def\arbk{\begin{picture}(4,5)\put(1,4){\circle*{0.7}}\put(2,3){\circle*{0.7}}\put(3,1){\circle*{0.7}}\put(4,2){\circle*{0.7}}\put(4,2){\Line(-1,-1)}\put(2,3){\Line(2,-1)}\put(1,4){\Line(1,-1)}\put(1,4){\circle*{1}}\end{picture}}

\def\arbl{\begin{picture}(4,5)\put(1,4){\circle*{0.7}}\put(2,2){\circle*{0.7}}\put(3,3){\circle*{0.7}}\put(4,2){\circle*{0.7}}\put(3,3){\Line(-1,-1)}\put(3,3){\Line(1,-1)}\put(1,4){\Line(2,-1)}\put(1,4){\circle*{1}}\end{picture}}

\def\arbm{\begin{picture}(4,4)\put(1,2){\circle*{0.7}}\put(2,3){\circle*{0.7}}\put(3,2){\circle*{0.7}}\put(4,1){\circle*{0.7}}\put(3,2){\Line(1,-1)}\put(2,3){\Line(-1,-1)}\put(2,3){\Line(1,-1)}\put(2,3){\circle*{1}}\end{picture}}

\def\arbn{\begin{picture}(4,5)\put(1,4){\circle*{0.7}}\put(2,3){\circle*{0.7}}\put(3,2){\circle*{0.7}}\put(4,1){\circle*{0.7}}\put(3,2){\Line(1,-1)}\put(2,3){\Line(1,-1)}\put(1,4){\Line(1,-1)}\put(1,4){\circle*{1}}\end{picture}}

\def\arbga{\begin{picture}(5,4)\cerp{.5}1\cerp21\cerp{3.5}1\cerp51
 \cerg{2.75}3
 \put(2.75,3){\Line(-2.25,-2)}
 \put(2.75,3){\Line(-.75,-2)}
 \put(2.75,3){\Line( .75,-2)}
 \put(2.75,3){\Line( 2.25,-2)}
\end{picture}}

\def\arbgb{\begin{picture}(5,5)\cerp{.5}0\cerp{.5}2\cerp22\cerp{3.5}2\cerg{2}4
 \put(2,4){\Line(-1.5,-2)}
 \put(2,4){\Line( 0,-2)}
 \put(2,4){\Line( 1.5,-2)}
 \put(.5,2){\Line( 0,-2)}
\end{picture}}

\def\arbgc{\begin{picture}(5,6)\cerp{1}1\cerp31\cerp23\cerp43\cerg35
 \put(3,5){\Line(-1,-2)}
 \put(3,5){\Line( 1,-2)}
 \put(2,3){\Line(-1,-2)}
 \put(2,3){\Line( 1,-2)}
\end{picture}}

\def\arbgd{\begin{picture}(5,5)\cerp{2}0\cerp{.5}2\cerp22\cerp{3.5}2\cerg{2}4
 \put(2,4){\Line(-1.5,-2)}
 \put(2,4){\Line( 0,-2)}
 \put(2,4){\Line( 1.5,-2)}
 \put(2,2){\Line( 0,-2)}
\end{picture}}

\def\arbge{\begin{picture}(5,6)\cerp{1.5}1\cerp{3}1\cerp{4.5}1\cerp{3}3\cerg{3}5
 \put(3,5){\Line( 0,-2)}
 \put(3,3){\Line( 0,-2)}
 \put(3,3){\Line( -1.5,-2)}
 \put(3,3){\Line( 1.5,-2)}
\end{picture}}

\def\arbgf{\begin{picture}(4,8)\cerp{1}1\cerp{1}3\cerp15\cerp{3}5\cerg{2}7
 \put(2,7){\Line(-1,-2)}
 \put(2,7){\Line( 1,-2)}
 \put(1,5){\Line( 0,-2)}
 \put(1,3){\Line( 0,-2)}
\end{picture}}

\def\arbgg{\begin{picture}(5,5)\cerp{3.5}0\cerp{.5}2\cerp22\cerp{3.5}2\cerg{2}4
 \put(2,4){\Line(-1.5,-2)}
 \put(2,4){\Line( 0,-2)}
 \put(2,4){\Line( 1.5,-2)}
 \put(3.5,2){\Line( 0,-2)}
\end{picture}}

\def\arbgh{\begin{picture}(5,8)\cerg{2}7\cerp{2}5\cerp13\cerp{3}3\cerp{1}1
 \put(2,7){\Line( 0,-2)}
 \put(2,5){\Line(-1,-2)}
 \put(2,5){\Line( 1,-2)}
 \put(1,3){\Line( 0,-2)}
\end{picture}}

\def\arbgi{\begin{picture}(5,6)\cerg{2}5\cerp{1}3\cerp33\cerp{1}1\cerp{3}1
 \put(2,5){\Line(-1,-2)}
 \put(2,5){\Line( 1,-2)}
 \put(1,3){\Line( 0,-2)}
 \put(3,3){\Line( 0,-2)}
\end{picture}}

\def\arbgj{\begin{picture}(5,6)\cerg{2}5\cerp{1}3\cerp33\cerp{2}1\cerp{4}1
 \put(2,5){\Line(-1,-2)}
 \put(2,5){\Line( 1,-2)}
 \put(3,3){\Line(-1,-2)}
 \put(3,3){\Line( 1,-2)}
\end{picture}}

\def\arbgk{\begin{picture}(4,8)\cerg{2}7\cerp{2}5\cerp23\cerp{1}1\cerp{3}1
 \put(2,7){\Line( 0,-2)}
 \put(2,5){\Line( 0,-2)}
 \put(2,3){\Line(-1,-2)}
 \put(2,3){\Line( 1,-2)}
\end{picture}}

\def\arbgl{\begin{picture}(4,8)\cerg{2}7\cerp{2}5\cerp13\cerp{3}3\cerp{3}1
 \put(2,7){\Line( 0,-2)}
 \put(2,5){\Line(-1,-2)}
 \put(2,5){\Line( 1,-2)}
 \put(3,3){\Line( 0,-2)}
\end{picture}}

\def\arbgm{\begin{picture}(4,8)\cerg{2}7\cerp{1}5\cerp35\cerp{3}3\cerp{3}1
 \put(2,7){\Line(-1,-2)}
 \put(2,7){\Line( 1,-2)}
 \put(3,5){\Line( 0,-2)}
 \put(3,3){\Line( 0,-2)}
\end{picture}}

\def\arbgn{\begin{picture}(2,10)\cerg{1}9\cerp{1}7\cerp15\cerp{1}3\cerp{1}1
 \put(1,9){\Line( 0,-2)}
 \put(1,7){\Line( 0,-2)}
 \put(1,5){\Line( 0,-2)}
 \put(1,3){\Line( 0,-2)}
\end{picture}}

\def\arbcga{\begin{picture}(5,6)
\cerp{1}1\cerp31\cerp23\cerp53\cerg{3.5}5\cerp51
 \put(3.5,5){\Line(-1.5,-2)}
 \put(3.5,5){\Line( 1.5,-2)}
 \put(2,3){\Line(-1,-2)}
 \put(2,3){\Line( 1,-2)}
 \put(5,1){\Line( 0, 2)}
\end{picture}}

\def\arbcgb{\begin{picture}(5,6)
\cerg{1.5}5\cerp{0}3\cerp33\cerp{2}1\cerp{4}1\cerp01
 \put(1.5,5){\Line(-1.5,-2)}
 \put(1.5,5){\Line( 1.5,-2)}
 \put(3,3){\Line(-1,-2)}
 \put(3,3){\Line( 1,-2)}
 \put(0,1){\Line( 0, 2)}
\end{picture}}

\def\arbcgc{\begin{picture}(5,8)\cerp{1}1\cerp31\cerp23\cerp43\cerp35\cerg37
 \put(3,7){\Line( 0,-2)}
 \put(3,5){\Line(-1,-2)}
 \put(3,5){\Line( 1,-2)}
 \put(2,3){\Line(-1,-2)}
 \put(2,3){\Line( 1,-2)}
\end{picture}}

\def\arbcgd{\begin{picture}(5,8)
\cerg{2}7\cerp25\cerp{1}3\cerp33\cerp{2}1\cerp{4}1
 \put(2,7){\Line( 0,-2)}
 \put(2,5){\Line(-1,-2)}
 \put(2,5){\Line( 1,-2)}
 \put(3,3){\Line(-1,-2)}
 \put(3,3){\Line( 1,-2)}
\end{picture}}

\def\arbcge{\begin{picture}(4,10)\
\cerg{2}9\cerp27\cerp{2}5\cerp23\cerp{1}1\cerp{3}1
 \put(2,9){\Line( 0,-2)}
 \put(2,7){\Line( 0,-2)}
 \put(2,5){\Line( 0,-2)}
 \put(2,3){\Line(-1,-2)}
 \put(2,3){\Line( 1,-2)}
\end{picture}}

\title[Quadri-algebras, preLie algebras, and Lie idempotents]{Quadri-algebras, preLie algebras,\\ and the Catalan family of Lie idempotents}

\author[L. Foissy, F. Menous, J.-C.~Novelli and J.-Y.~Thibon]%
{Lo\"ic Foissy, Fr\'ed\'eric Menous,\\ Jean-Christophe Novelli and Jean-Yves
Thibon}

\address[Foissy]{
F\'ed\'eration de Recherche Math\'ematique du Nord Pas de Calais FR 2956\\
Laboratoire de Math\'ematiques Pures et Appliqu\'ees Joseph Liouville\\
Universit\'e du Littoral C\^ote d'Opale-Centre Universitaire de la Mi-Voix\\ 
50, rue Ferdinand Buisson, CS 80699,  62228 Calais Cedex, France}

\address[Menous]{Laboratoire de Math\'ematiques\\
B\^at. 425\\
Universit\'e Paris-Sud\\
91405 Orsay Cedex\\
France}

\address[Novelli, Thibon] {Laboratoire d'Informatique Gaspard-Monge, Universit\'e Gustave Eiffel, CNRS, ENPC, ESIEE-Paris,\\
5 Boulevard Descartes \\Champs-sur-Marne \\77454 Marne-la-Vall\'ee cedex 2 \\
FRANCE}

\email[Lo\"ic Foissy]{foissy@univ-littoral.fr}
\email[Fr\'ed\'eric Menous]{Frederic.Menous@u-psud.fr }
\email[Jean-Christophe Novelli]{novelli@univ-mlv.fr}
\email[Jean-Yves Thibon]{jyt@univ-mlv.fr} 
\date{\today}

\begin{document}

\begin{abstract}
We compute the expansion of the Catalan family of Lie idempotents introduced
in [Menous et al., Adv. Applied Math. 51 (2013), 177-22] on the PBW basis of
the Lie module.
It is found that the coefficient of a tree depends only on its number of left
and right internal edges. In particular, the Catalan idempotents belong to a
preLie algebra based on naked binary trees, of which we identify several Lie
and preLie subalgebras.
\end{abstract}

\maketitle

\section{Introduction}

A \emph{Lie idempotent} is an idempotent of the group algebra of the symmetric
group which acts on the free associative algebra as a projector onto the free
Lie algebra~\cite{Re}.

Historically, the first example of a Lie idempotent is provided by Dynkin's
theorem (1947,\cite{Dyn1}):
The linear map 
$\Theta_n:\ a_1a_2\cdots a_n \mapsto[\cdots [a_1,a_2],\cdots ],a_n],$
sending a word to its iterated bracketing, satisfies
$\Theta_n\circ \Theta_n=n\Theta_n$.
This result, originally intended as a mean of expanding the Hausdorff series 
$H(a_1,\ldots,a_N)=\log(e^{a_1}\cdots e^{a_N}) $in terms of
commutators, can also be used to give a new proof of the fact that $H$ is a
Lie series, and to prove the Poincar\'e-Birkhoff-Witt theorem. In particular
one can deduce from it most basic facts about free Lie algebras, such as
Friedrich's criterion \cite{Cartier}.

In 1969, Solomon \cite{So1} introduced another Lie idempotent, with the aim of
providing a constructive proof of the Poincar\'e-Birkhoff-Witt theorem, that
is, an explicit isomorphism between the universal enveloping algebra of a Lie
algebra and its symmetric algebra.
It turns out that this idempotent is also related to the Hausdorff series: it
computes its expansion in terms of words, and has been rediscovered several
times in this context (see, {\it e.g.},\cite{BMP,MP}).
Solomon's idempotent is also known as the first Eulerian idempotent
\cite{Lod89}.

Finally, Witt's formulas \cite{Witt} for the dimensions of the
multihomogeneous components of the free Lie algebra $L(V)$ generated by a
vector space $V$ show that, as a representation of $GL(V)$, $L_n(V)$ is
isomorphic to the eigenspace with eigenvalue $e^{2i\pi/n}$ of the cyclic shift
operator acting on $V^{\otimes n}$. An explicit interwining operator for these
representations is provided by Klyachko's idempotent \cite{Kl}, introduced in
1974.

Prior to the introduction of noncommutative symmetric functions \cite{NCSF1},
these three examples, which are given by very different expressions, were the
only known Lie idempotents. They have however one important common point: they
all belong to the \emph{descent algebra}, a remarkable subalgebra of the group
algebra of the symmetric group, introduced by Solomon \cite{Sol} in 1976. The
descent algebra is a noncommutative version of the character ring of the
symmetric group, and noncommutative symmetric functions are to the descent
algebra what ordinary symmetric functions are to the character ring.  In
particular, on can lift to the descent algebra various operations such as the
$(1-q)$-transform and its inverse \cite{NCSF2}. This allowed to give a
complete characterization of Lie idempotents in the descent algebra, and to
provide an explicit interpolation between all known examples.

It came therefore as a surprise that a sequence of operators defined in the
context of resurgence theory \cite{Men}, when interpreted as noncommutative
symmetric functions, thanks to an isomophism with \'Ecalle's Hopf algebra of
alien operators \cite{MNT}, provided a new family of Lie idempotents of the
descent algebras, not explained by the previous constructions. 

In the sequel, we shall give a much simpler expression of these idempotents,
in terms of the recently introduced PBW basis of the Lie module \cite{ST}. It
is however not clear on this expression that the result belongs to the descent
algebra. In the process of investigating the symmetries of Lie idempotents in
this basis, we have been led to the discovery of various Lie and preLie
subalgebras of the convolution Lie algebra of permutations. The PBW expansion
of the Solomon idempotent, determined by Bandiera and Sch\"atz \cite{BS}, also
satifies the same type of symmetry. It involves a Lie algebra of binary trees,
of which we provide an alternative interpretation. This Lie algebra have
remarkable subalgebras, whose bases are parametrized by families of bicolored
trees. We prove that Lie elements of the descent algebras all belong to one of
these subalgebras and conjecture a similar statement for Lie elements of the
Loday-Ronco algebra.

\section{Background}

\subsection{Lie idempotents}

Let $V$ be a vector space over some field $\K$ of characteristic~0.
Let $T(V)$ be its tensor algebra, and $L(V)$ the free Lie algebra
generated by $V$. We denote by $L_n(V)=L(V)\cap V^{\otimes n}$ its homogeneous
component of degree $n$.
The group algebra $\K\SG_n$ of the symmetric group acts on the right on
$V^{\otimes n}$ by
\begin{equation}
(v_1\otimes v_2\otimes\cdots\otimes v_n)\cdot\sigma
=
v_{\sigma(1)}\otimes v_{\sigma(2)}\otimes\cdots\otimes v_{\sigma(n)}\,.
\end{equation}

This action commutes with the left action of $GL(V)$,
and when $\dim\, V \ge n$, which we shall usually assume, these actions are
the commutant of each other
(Schur-Weyl duality).

Any $GL(V)$-equivariant projector $\Pi_n:\ V^{\otimes n}\rightarrow L_n(V)$
can therefore be regarded as an idempotent $\pi_n$ of $\K\SG_n$:
$\Pi_n({\bf v})={\bf v}\cdot\pi_n$. By definition
(cf. \cite{Re}), such an element is called a {\em Lie idempotent}
whenever its image is $L_n(V)$. 
Then, a homogeneous element $P_n\in V^{\otimes n}$ is in $L_n(V)$
if and only if $P_n\pi_n=P_n$.

From now on, we fix a basis $A=\{a_1,a_2,\ldots\}$ of $V$. We identify
$T(V)$ with the free associative algebra $\K\<A\>$, and $L(V)$
with the free Lie algebra $L(A)$.

\subsection{Convolution algebras}

For $\sigma\in\SG_n$, let $g_\sigma$ be the endomorphism of $\K_n\<A\>$
defined by
\begin{equation}
g_\sigma(a_1 a_2\cdots a_n)
 = a_1a_2\cdots a_n\cdot\sigma
 = a_{\sigma(1)}a_{\sigma(2)}\cdots a_{\sigma(n)}.
\end{equation}
The free algebra $\K\<A\>$ is a graded bialgebra for the coproduct 
\begin{equation}
\Delta(a) = a\otimes 1 + 1\otimes a,\quad (a\in A),
\end{equation}
so that a convolution product on the space of graded
endomorphisms $\End_{\gr}(\K\<A\>)$ can be defined by
\begin{equation}
f\star g (w) = \mu\circ(f\otimes g)\circ\Delta(w),
\end{equation}
where $\mu$ denotes the multiplication of $\K\<A\>$ ({\it i.e.}, the
concatenation product).

For permutations, this operation reads
\begin{equation}\label{eq:star}
g_\alpha\star g_\beta
 = \sum_{\gf{\gf{\std(u)=\alpha}{\std(v)=\beta}}{\gamma=uv}}g_\gamma,
\end{equation}
where $\std(u)$ denotes the standardization of the word $u$, that is, the
unique permutation having the same inversions as $u$.

We also define the $\star$ product on permutations by the same formula.
This convolution has been extensively studied by Reutenauer \cite{Re86,Re}
while investigating free Lie algebras and Lie idempotents. 

When $A$ is finite, the $g_\sigma$ are not linearly independent anymore when
the size of~$\sigma$ exceeds the cardinality of $A$. By taking an appropriate
inverse limit over an increasing sequence of alphabets, one obtains a
convolution algebra $\K\SG$ based on all permutations, and it has been shown
by Malvenuto and Reutenauer \cite{MR} that this algebra acquires a Hopf
algebra structure when endowed with the coproduct
\begin{equation}
\Delta \sigma = \sum_{u,v}\<\sigma,u\shuffle v\>{\std(u)}\otimes {\std(v)},
\end{equation}
where $\<\sigma,u\shuffle v\>$ denotes the coefficient of $\sigma$ in the
shuffle product $u\shuffle v$ of the words $u$ and $v$.

\subsection{Free quasi-symmetric functions}

The Malvenuto-Reutenauer Hopf algebra admits a convenient
``polynomial realization''.

The algebra of \emph{free quasi-symmetric functions} over a totally ordered
alphabet $A$, denoted by $\FQSym(A)$, is the algebra spanned by the
noncommutative ``stable polynomials'' (formal series of bounded degree,
defined for any totally ordered alphabet)
\begin{equation}
\G_\sigma(A)  := \sum_{\gf{w\in A^n}{\std(w)=\sigma}} w,
\end{equation}
where $\sigma$ is a permutation in the symmetric group $\SG_n$ \cite{NCSF6}.

Its multiplication rule turns out to be given by 
\begin{equation}
\G_\alpha \G_\beta = \sum_{
\gf{\std(u)=\alpha}{\gf{\std(v)=\beta}{\gamma=uv}}}
\G_\gamma,
\end{equation}
as in the Malvenuto-Reutenauer algebra $\K\SG$. Moreover, its natural coproduct,
defined by the ordinal sum $A\oplus B$ of two mutually commuting alphabets
\begin{equation}
\Delta \G_\sigma := 
\G_\sigma(A\oplus B)=
\sum_{u,v}\<\sigma,u\,\shuffle\, v\>\G_{\std(u)}\otimes \G_{\std(v)},
\end{equation}
also coincides with the Malvenuto-Reutenauer coproduct (identifying
$F(A)G(B)$ with $F\otimes G$).

Thus, $\FQSym$ is isomorphic to $\K\SG$ as a Hopf algebra. 

In the sequel, we shall often identify permutations $\sigma$ with the
corresponding $\G_\sigma$ without further notice.

\subsection{Dendriform structure}

A dendriform algebra~\cite{Lod0,Lod} is an associative algebra
$({\mathcal A},\cdot)$ 
endowed with two bilinear operations $\gaudend$, $\droitdend$,
such that 
\begin{equation}
a\cdot b = a \gaudend b + a \droitdend b\,,
\end{equation}
satisfying the relations
\begin{equation}
(x\gaudend y)\gaudend z = x\gaudend (y\cdot z)\,,
(x\droitdend y)\gaudend z = x\droitdend (y\gaudend z)\,,
(x\cdot y)\droitdend z = x\droitdend (y\droitdend z)\,.
\end{equation}

The dendriform structure of $\FQSym$ is inherited from that of the free
associative algebra over $A$, which is~\cite{NT1,NT2}
\begin{eqnarray}
u\gaudend v=
\begin{cases}
uv &\mbox{if $\max(v) < \max(u)$}\\
0 &\mbox{otherwise},\\
\end{cases}\\
u\droitdend v=
\begin{cases}
uv &\mbox{if $\max(v)\geq\max(u)$}\\
0 &\mbox{otherwise}.\\
\end{cases}
\end{eqnarray}

This yields
\begin{equation}
\G_\alpha \G_\beta = \G_\alpha \gaudend \G_\beta + \G_\alpha \droitdend
\G_\beta\,,
\end{equation}
where
\begin{equation}\label{eq:gauche}
\G_\alpha \gaudend \G_\beta =
\sum_{\gf{\gamma=uv \in \alpha\star \beta}{|u|=|\alpha| ;\, \max(v)<\max(u)}}
  \G_\gamma\,,
\end{equation}
\begin{equation}\label{eq:droit}
\G_\alpha \droitdend \G_\beta =
\sum_{\gf{\gamma=u.v\in \alpha\star\beta}{|u|=|\alpha| ;\, \max(v)\geq\max(u)}}
   \G_\gamma\,.
\end{equation}
where $\alpha\star\beta$ is interpreted as the set of permutations occuring in the convolution.
Then $x=\G_1$ generates a free dendriform algebra in $\FQSym$, isomorphic to
$\PBT$, the Loday-Ronco algebra of planar binary trees~\cite{LR1}.

\subsection{Quadri-algebras}

The half-products of $\FQSym$, which are defined by splitting the
concatenation of two words according to the position of the greatest letter
can again be refined according to the position of the smallest one:
for $\alpha\in\SG_k$, $\beta\in\SG_l$ and $n=k+l$, 

\begin{align}
\G_\alpha \nwarrow \G_\beta
  & = \sum_{\gf{\gamma=uv \in \alpha\star\beta}{1, n\in u}} \G_\gamma\,,\\
\G_\alpha \swarrow \G_\beta
  & = \sum_{\gf{\gamma=uv \in \alpha\star\beta}{1\in v,\ n\in u}} \G_\gamma\,,\\
\G_\alpha \searrow \G_\beta
  & = \sum_{\gf{\gamma=uv \in \alpha\star\beta}{1,n\in v}} \G_\gamma\,,\\
\G_\alpha \nearrow \G_\beta
  & = \sum_{\gf{\gamma=uv \in \alpha\star\beta}{1\in u,n\in v}} \G_\gamma\,.
\end{align}
The relations satisfied by these partial products have led Aguiar and Loday to
the notion of \emph{quadri-algebra}~\cite{AL}.

A quadri-algebra is a family $(A,\nwarrow,\swarrow,\searrow,\nearrow)$, where
$A$ is a vector space and $\nwarrow$, $\swarrow$, $\searrow$, $\nearrow$ are
products on $A$, such that for all $x,y,z \in A$:

{\footnotesize
\begin{align}
  (x\nwarrow y)\nwarrow z     & =x \nwarrow (y\star z),
 &(x\nearrow y) \nwarrow z    & =x \nearrow (y\leftarrow z),
 &(x \uparrow y) \nearrow z   & =x \nearrow (y \rightarrow z),\\
  (x\swarrow y)\nwarrow z     & =x \swarrow (y\uparrow z),
 &(x\searrow y) \nwarrow z    & =x \searrow (y\nwarrow z),
 &(x \downarrow y) \nearrow z & =x \searrow (y \nearrow z),\\
  (x\leftarrow y)\swarrow z   & =x \swarrow (y\downarrow z),
 &(x\rightarrow y) \swarrow z & =x \searrow (y\swarrow z),
 &(x \star y) \searrow z      & =x \searrow (y \searrow z),
\end{align}
where:
\begin{align}
\leftarrow  & = \nwarrow+\swarrow, &
\rightarrow & = \nearrow+\searrow, &
\uparrow    & = \nwarrow+\nearrow, &
\downarrow  & = \swarrow+\searrow,
\end{align} 
\begin{equation}
\star = \nwarrow+\swarrow+\searrow+\nearrow
      = \leftarrow+\rightarrow
      = \uparrow+\downarrow.
\end{equation}
}

The augmentation ideal $\FQSym_+$ of the Hopf algebra $\FQSym$ is a
quadri-algebra for these operations \cite{AL}, up to the involution
$\sigma \mapsto \sigma^{-1}$. 

As $\FQSym$ is self-dual, its coproduct can also be split into four parts,
making it a quadri-coalgebra.
Dualizing, with the pairing defined by
$\langle \G_\sigma,\G_\tau\rangle=\delta_{\sigma,\tau^{-1}}$, 
we obtain a quadri-coalgebra structure on $\FQSym_+$, defined by: 
\begin{align}
\Delta_\nwarrow(\G_\sigma)
  &= \sum_{\sigma(1),\sigma(n) \leq i<n}
    \G_{\sigma_{\{1,\ldots,i\}}}\otimes \G_{\std(\sigma_{\{1+i,\ldots,n\}})}\\
\Delta_\swarrow(\G_\sigma)
  &= \sum_{\sigma(n) \leq i < \sigma(1)}
    \G_{\sigma_{\{1,\ldots,i\}}}\otimes \G_{\std(\sigma_{\{1+i,\ldots,n\}})}\\
\Delta_\searrow(\G_\sigma)
  &= \sum_{1\leq i< \sigma(1) ,\sigma(n)}
    \G_{\sigma_{\{1,\ldots,i\}}}\otimes \G_{\std(\sigma_{\{1+i,\ldots,n\}})}\\
\Delta_\nearrow(\G_\sigma)
 &= \sum_{\sigma(1) \leq i <\sigma(n)}
    \G_{\sigma_{\{1,\ldots,i\}}}\otimes \G_{\std(\sigma_{\{1+i,\ldots,n\}})},
\end{align}
where for all $I\subseteq \{1,\ldots,n\}$, $\sigma_I$ is the word obtained by
deleting in $\sigma$ the letters which do not belong to $I$.
The sum of the four coproducts is the ususal coproduct $\tdelta$:
\begin{align}
\tdelta(\G_\sigma)
& = \sum_{1 \leq i <n}
    \G_{\sigma_{\{1,\ldots,i\}}}\otimes \G_{\std(\sigma_{\{1+i,\ldots,n\}})}.
\end{align}
Moreover, for any $a,b\in \FQSym$:
\begin{align}
\tdelta(a\nwarrow b)
 & = a'\uparrow b\otimes a''+a'\otimes a''\leftarrow b
    +a'\uparrow b'\otimes a''\leftarrow b'',\\
\tdelta(a\swarrow b)
 & = b\otimes a+a'\downarrow b\otimes a''+b'\otimes a\leftarrow b''
    +a'\downarrow b'\otimes a''\leftarrow b'',\\
\tdelta(a\searrow b)
 & = a\downarrow b'\otimes b''+b'\otimes a\rightarrow b''
    +a'\downarrow b'\otimes a''\rightarrow b'',\\
\tdelta(a\nearrow b)
 & = a\otimes b+a\uparrow b'\otimes b''+a'\otimes a''\rightarrow b
    +a'\uparrow b'\otimes a''\rightarrow b''.
\end{align}

As a consequence, if $a$ and $b$ are two primitive elements of $\FQSym$:
\begin{align}
\tdelta(a\nearrow b-b\swarrow a)&=a\otimes b-a\otimes b=0,
\end{align}
so $a\nearrow b-b\swarrow a$ is also primitive.

\begin{proposition}
The operation
\begin{equation}
x\nesw y = x\nearrow y - y \swarrow x
\end{equation}
preserves primitive elements.
\qed
\end{proposition}


\subsection{Solomon's descent algebra}

The descent algebras have been introduced by Solomon \cite{Sol} for general
finite Coxeter groups in the following way. Let $(W,S)$ be a Coxeter system.
One says that $w\in W$ has a descent at $s\in S$ if $w$ has a reduced word
ending by $s$.
For $W=\SG_n$ and $s_i=(i,i+1)$, this means that $w(i)>w(i+1)$, whence the
terminology. In this case, we rather say that $i$ is a descent of $w$.
Let $\Des(w)$ denote the descent set of $w$, and for a subset $E\subseteq S$,
set
\begin{equation}
D_E = \sum_{\Des(w)=E} w \ \ \in \ZZ W \,.
\end{equation}
Solomon has shown that the $D_E$ span a $\ZZ$-subalgebra $\Sigma_n$
of $\ZZ W$. Moreover,
\begin{equation}
D_{E'}D_{E''}=\sum_E c^E_{E'E''} D_E,
\end{equation}
where the coefficients $c^E_{E'E''}$ are nonnegative integers.

In the case of $\SG_n$, it is convenient to encode descent sets by
compositions of $n$. If $E=\{d_1,\ldots,d_{r-1}\}$, we set $d_0=0$, $d_r=n$
and $I=C(E)=(i_1,\ldots,i_r)$, where $i_k=d_k-d_{k-1}$.
We also say that $E$ is the descent set of $I$. From now on, we shall write
$D_I$ instead of $D_E$. We shall also write $C(\sigma)=I$ if the descent set of
$\sigma$ is $E$.

Most known examples of Lie idempotents turn out to belong to the descent
algebra $\Sigma_n$.
The simplest (and oldest) example is the \emph{Dynkin idempotent} \cite{Dyn1}
\begin{equation}\label{eq:dyn}
\theta_n
 = \frac1n [\ldots[[[1,2],3],\ldots,],n]
 = \frac1n \sum_{k=0}^{n-1}(-1)^kD_{1^k,n-k}.
\end{equation}
The Solomon idempotent \cite{So1}, or first Eulerian idempotent is
\begin{equation}\label{eq:sol}
\phi_n
  = \frac1n
    \sum_{\sigma\in\SG_n}
        \frac{(-1)^{d(\sigma)}}{\binom{n-1}{d(\sigma)}} \sigma
\end{equation}
and the Klyachko idempotent \cite{Kl} is
\begin{equation}
\kappa_n
  = \frac1n
    \sum_{\sigma\in\SG_n}\omega^{\maj(\sigma)}\sigma,
\end{equation}
where $d(\sigma)$ is the number of descents of $\sigma$, $\maj(\sigma)$ is their sum,
and $\omega=e^{2i\pi/n}$.
Note that on these expressions, one easily sees that both the Solomon
idempotent and the Klyachko idempotent live inside the descent algebra since
the coefficient of a permutation $\sigma$ only depends on its descents.

\medskip
There is a one-parameter family (the $q$-Eulerians) interpolating between
these three examples \cite{NCSF2,Cha1}, and more recently, another
one-parameter family (the Catalan family) related to mould calculus and random
walks on the line has been introduced \cite{MNT}. This family is given by an
explicit expansion in terms of descent classes. One of the aims of the present
paper is to give its expansion on the so-called Poincar\'e-Birkhoff-Witt
basis, to be defined below.

\subsection{Noncommutative symmetric functions}

The algebra of ordinary symmetric functions $Sym$ can be regarded as the free
associative and commutative algebra over an infinite sequence $(h_n)_{n\ge 1}$
of homogeneous generators ($h_n$ is of degree $n$), so that its linear bases
in degree $n$ are naturally labelled by partitions of $n$
(e.g., products $h_\mu=h_{\mu_1}\cdots h_{\mu_r}$ of complete homogeneous
functions) \cite{Mcd}.

Similarly, the algebra $\Sym$ of noncommutative symmetric functions is the
free associative (but noncommutative) algebra over an infinite sequence
$(S_n)_{n\ge 1}$ of homogeneous generators, endowed with a natural
homomorphism $S_n\mapsto h_n$ (commutative image) \cite{NCSF1}.

Thus, linear bases of the homogeneous component $\Sym_n$ of $\Sym$ are
labelled by compositions of $n$, exactly as those of the descent algebra
$\Sigma_n$ of $\SG_n$. 

Noncommutative symmetric functions can be realized in terms of an auxiliary
(totally ordered) alphabet $A=\{a_1,a_2,\ldots\}$ by setting
\begin{equation}
S_n(A)
 = \sum_{i_1\le i_2\le\ldots\le i_n}a_{i_1}a_{i_2}\cdots a_{i_n}
 = \G_{12\cdots n}(A)\,,
\end{equation}
that is, $S_n$ is the sum of nondecreasing words, or, otherwise said, words
with no descent. Then, obviously,
\begin{equation}
S^I= S_{i_1}S_{i_2}\cdots S_{i_r}
\end{equation}
is the sum of words whose descent set is contained in $\Des(I)$ (the descents
of a word are defined as for permutations as those $i$ such that
$w_i>w_{i+1}$, and are similarly encoded as compositions $C(w)$ of $n$).
Since $S_n(A)=\G_{12\cdots n}(A)$, $\Sym(A)$ is actually a subalgebra of
$\FQSym(A)$.

Introducing the \emph{noncommutative ribbon Schur functions}
\begin{equation}
R_I(A)=\sum_{C(w)=I}w = \sum_{C(\sigma)=I}\G_\sigma(A)
\end{equation}
whose commutative image are indeed the skew Schur functions indexed by
ribbon diagrams, we have
\begin{equation}
S^I =\sum_{J\le I}R_J
\end{equation}
where $J\le I$ is the \emph{reverse refinement order}, which means
that $\Des(J)\subseteq\Des(I)$.

The linear map defined by $\alpha:\, D_I \rightarrow R_I$ appears therefore as
a natural choice for a correspondence $\Sigma_n\rightarrow \Sym_n$.
This choice is actually \emph{canonical}. Indeed, there is a natural way
to introduce an \emph{internal product} $*$ on $\Sym$, by dualizing a natural
coproduct of its graded dual (which is the so-called Hopf algebra of
quasi-symmetric functions).

For this product, $\alpha$ is an \emph{anti-isomorphism}. 
This is convenient, because we want to interpret permutations as endomorphisms
of tensor algebras: if $g_\sigma(w)=w\sigma$, then
$g_\sigma\circ g_\tau=g_{\tau\sigma}$.

The \emph{internal product} $*$ is consistently defined on $\FQSym$ by
\begin{equation}
\G_\sigma * \G_\tau = 
\begin{cases} 
\G_{\tau\sigma} \ \mbox{\rm if $|\tau|=|\sigma|$} \\
0 \ \mbox{\rm otherwise}
\end{cases}
\end{equation}
where $\tau\sigma$ is the usual product of the symmetric group.

The fundamental property for computing with the internal product in $\Sym$ is
the following \emph{splitting formula}:
\begin{proposition}[\cite{NCSF1}]
\label{mackey}
Let $F_1,F_2,\ldots,F_r,G \in \Sym$. Then, 
\begin{equation}
(F_1F_2\cdots F_r)*G = \mu_r \left[ (F_1\otimes\cdots\otimes F_r)* \Delta^r G
\right]
\end{equation}
where in the right-hand side, $\mu_r$ denotes the $r$-fold ordinary
multiplication and $*$ stands for the operation induced on $\Sym^{\otimes n}$
by $*$.
\end{proposition}                       

\subsection{Lie idempotents as noncommutative symmetric functions}

The first really interesting question about noncommutative symmetric functions
is perhaps ``what are the noncommutative power sums?''.
Indeed, the answer to this question is far from being unique.

If one starts from the classical expression
\begin{equation}\label{eq:exp}
\sigma_t(X)
 = \sum_{n\ge 0}h_n(X)t^n
 = \exp\left\{ \sum_{k\ge 1} p_k \frac{t^k}{k} \right\}\,,
\end{equation}
one can choose to define noncommutative power sums $\Phi_k$ by the same
formula
\begin{equation}
\sigma_t(A)
 = \sum_{n\ge 0}S_n(A)t^n
 = \exp\left\{ \sum_{k\ge 1} \Phi_k \frac{t^k}{k} \right\}\,,
\end{equation}
but a noncommutative version of the Newton formulas
\begin{equation}
nh_n = h_{n-1}p_1+ h_{n-2}p_2+\cdots+p_n
\end{equation}
which are derived by taking the logarithmic derivative of (\ref{eq:exp})
leads to different noncommutative power-sums $\Psi_k$ inductively defined by
\begin{equation}\label{eq:newtonpsi}
nS_n = S_{n-1}\Psi_1+S_{n-2}\Psi_2+\cdots+\Psi_n \,.
\end{equation}

A bit of computation reveals then that
\begin{equation}\label{eq:pnhook}
\Psi_n = R_n - R_{1,n-1}+ R_{1,1,n-2}- \cdots
= \sum_{k=0}^{n-1}(-1)^k R_{1^k,n-k}\,,
\end{equation}
which is analogous to the classical expression of $p_n$ as the alternating
sum of hook Schur functions. Therefore, in the descent algebra,
$\Psi_n$ corresponds to Dynkin's element, $n\theta_n$ (see \eqref{eq:dyn}).

The $\Phi_n$ can also be expressed on the ribbon basis without much
difficulty, and one finds
\begin{equation}
\Phi_n =
  \sum_{|I|=n}
     \frac{(-1)^{l(I)-1}}{\binom{n-1}{l(I)-1}} R_I,
\end{equation}
so that $\Phi_n$ corresponds to $n\phi_n$ (see \eqref{eq:sol}).

\medskip
The case of Klyachko's idempotent is even more interesting: 
it is related to the so-called $(1-q)$-transform (see \cite{NCSF2}). 

The following result is proved in \cite{NCSF1}:
\begin{theorem} \label{LIEIDEM}
Let $F=\alpha(\pi)$ be an element of $\Sym_n$, where $\pi\in\Sigma_n$.
The following assertions are equivalent:
\begin {enumerate}
\item $\pi$ is a Lie quasi-idempotent;
\item $F$ is a primitive element for $\Delta$;
\item$F$ belongs to the Lie algebra $L(\Psi)$ generated by the $\Psi_n$.
\end{enumerate}
Moreover, $\pi$ is a Lie idempotent iff $F-\frac{1}{n}\Psi_n$ is in
the Lie ideal $[L(\Psi)\, ,\, L(\Psi)]$.
\end{theorem}

Thus, Lie idempotents are essentially the same thing as ``noncommutative power
sums'' (up to a factor $n$), and we shall from now on identify both notions:
a Lie idempotent in $\Sym_n$ is a primitive element whose commutative image
is $p_n/n$.

\medskip
Noncommutative symmetric functions can also be regarded as elements of the
free dendriform algebra on one generator, and Lie idempotents can be expressed
in terms of its preLie structure.
Indeed, there is a Hopf embedding $\iota:\ \Sym\rightarrow \PBT$ of
noncommutative symmetric functions into $\PBT$~\cite{NCSF1,NCSF6,HNT}, which
is given by
\begin{equation}
\iota(S_n)=
(\dots((x\droitdend x)\droitdend x)\dots)\droitdend x\quad\text{($n$ times).}
\end{equation}
For example, the Dynkin elements are
\begin{equation}
\iota(\Psi_n)=
(\dots((x\triangleright x)\triangleright x)
\dots\triangleright x \quad\text{($n$ times),}
\end{equation}
where
\begin{equation}\label{eq:prelie} 
a\triangleright b = a\droitdend b -b\gaudend a 
\end{equation}
is the preLie product defined in any dendriform algebra.

Using the embedding in $\FQSym$, the proof of this identity is remarkably
simple. Indeed,
\begin{equation}
\G_\sigma\droitdend x = \G_{\sigma\cdot(n+1)}\ \text{and}\
x \gaudend \G_\sigma = \G_{(n+1)\cdot \sigma}\,,
\end{equation}
so that $\iota(S_n)=\G_{12\dots n}$.
In terms of permutations, this is therefore the standard embedding of $\NCSF$
into $\FQSym$ as the descent algebra, for which, identifying $\G_\sigma$ with
$\sigma$,
\begin{equation}
\Psi_n = [[\dots,[1,2],\dots,n-1],n].
\end{equation}
It is then clear that
\begin{align}
x\triangleright x
  & = \G_{12}-\G_{21} = R_2-R_{11}=\Psi_2\,\\
\Psi_2\triangleright x
  & = \G_{123}-\G_{213}-\G_{312}+\G_{321} = R_3-R_{12}+R_{111}=\Psi_3\,\\
\Psi_{n-1}\triangleright x
  & = \G_{12\dots n}-\dots\pm \G_{n\dots 21}
    = \sum_{k=0}^{n-1}(-1)^kR_{1^k,n-k}=\Psi_n\,.
\end{align}

\subsection{The Catalan family}

Recently, Theorem \ref{LIEIDEM} has been used to identify a new family of Lie
idempotents of the descent algebras \cite{MNT}. Families of coefficients
obtained from iterated integrals coming from resurgence theory turned out to
be interpretable in terms of noncommutative symmetric functions, thanks to a
highly nontrivial isomorphim with a certain Hopf algebra of \'Ecalle's alien
operators.

Consider the generating series
\begin{equation}
\ca (a, b, t)
 = \frac{1 - (a + b) t - \sqrt{1 - 2 (a + b) t + (b - a)^2 t^2}}{2 a b t}
 = \sum_{n \geq 1} \ca_n (a, b) t^n .
\end{equation}
The coefficients $\ca_n(a,b)$ are homogeneous and symmetric polynomials
in $a$, $b$ of degree $n-1$:
\begin{equation}
\begin{array}{rcc}
\ca_1(a,b)&=&1 \\
\ca_2(a,b)&=&a+b \\
\ca_3(a,b)&=&a^2+3ab+b^2 \\
\ca_4(a,b)&=&a^3+6 a^2b+6ab^2+b^3 \\
\ca_5(a,b)&=&a^4+10a^3b+20a^2b^2+10ab^3+b^4 
\end{array}
\end{equation}
The coefficients of these polynomials are the Narayana numbers
\begin{equation}
T(n,k)=\frac1k \binom{n-1}{k-1} \binom{n}{k-1},
\end{equation}
so that
\begin{equation}
\ca_n(a,b)=\sum_{i=0}^{n-1}T(n,i+1)a^ib^{n-1-i}.
\end{equation}

For any sequence of signs
$\underline{\varepsilon} =(\varepsilon_1,...,\varepsilon_{n})$ ($n\geq 1$),
consider its minimal decomposition into stacks of identical signs 
\begin{equation}
\underline{\varepsilon}= (\varepsilon_1, \dots, \varepsilon_n) =
     (\eta_1)^{n_1} \dots (\eta_s)^{n_s},
\end{equation}
with $\eta_i \not= \eta_{i + 1}$ and $n_1 + \dots + n_s = n$. 

Define the \emph{signed ribbons} as
\begin{equation}
R_{\underline{\varepsilon} \bullet}=(-1)^{l(I)-1}R_I\quad
(R_\emptyset=1,\ R_{\bullet}=R_1),
\end{equation}
where $I$ is the composition such that
\begin{equation}
\Des(I)=\{1\leq i \leq n-1 \ ;\ \varepsilon_i=-\}.
\end{equation}

The following result is proved in \cite{MNT}:
\begin{theorem}
\label{thcat}
For $n\geq 1$, the element of $\Sym_{n+1}$ defined by
\begin{equation}
D_{a,b}^{n+1}=\sum_{
\underline{\varepsilon} 
= (\eta_1)^{n_1} \dots (\eta_s)^{n_s} }
  \left( \prod_{\gf{\eta_i = +}{i < s}} a \right)
  \left( \prod_{\gf{\eta_i = -}{i < s}} b \right)
  {\ca}_{n_1} (a, b) \dots {\ca}_{n_s} (a, b)
   R_{\underline{\varepsilon}\bullet}
\end{equation}
is primitive, and corresponds (up to a scalar factor) to a Lie idempotent of
the descent algebra.
\end{theorem}

For example, using the correspondence with the usual noncommutative
ribbon Schur functions,
\begin{equation}
\begin{array}{rcl} 
D_{a,b}^2 &=& R_2 -R_{11} \\
D_{a,b}^3 &=& (a+b)R_3 -a R_{21} -b R_{12}+(a+b)R_{111} \\
D_{a,b}^4 &=& (a^2+3ab+b^2) R_4 -a(a+b)R_{31}-
abR_{22}-(a+b)bR_{13} \\
 &&+a(a+b)R_{211}+abR_{121}+(a+b)bR_{112}-(a^2+3ab+b^2)R_{1111}\,.
\end{array}
\end{equation}

\section{The Lie module and its bases}

\subsection{The Lie module}

Lie idempotents belong to a subspace of $\K\SG_n$ called the {\it Lie module},
denoted by $\Lie(n)$. It is the linear span of the complete bracketing of
permutations, regarded as words over the alphabet $\{1,2,\ldots,n\}$.

Of course, these bracketings are not linearly independent, and the dimension
of $\Lie(n)$ is actually $(n-1)!$. The simplest basis of $\Lie(n)$ is the
Dynkin basis
\begin{equation}
D_\sigma = [\cdots [[1,\sigma(2)],\sigma(3)],\cdots,\sigma(n)]
\end{equation}
parametrized by permutations fixing 1, or equivalently
\begin{equation}
D'_\sigma = [\sigma(1),[\sigma(2),\cdots,[\sigma(n-1),n]\cdots]]
\end{equation}
parametrized by permutations fixing $n$. The expansions of various
Lie idempotents in this basis can be found in \cite{NCSF2}. 

\subsection{The Poincar\'e-Birkhoff-Witt basis}

A less obvious basis is the so-called Poincar\'e-Birkhoff-Witt basis
\cite{ST,BS}. Its elements are complete bracketings of permutations,
represented by complete binary trees with labelled leaves, such that for each
internal node, the smallest label is in the left subtree, and the greatest
label is in the right subtree.

Such a labelling of the leaves will be said to be {\it admissible}.
In particular, the leftmost leaf is always 1, and the rightmost one always $n$.

For example, the admissible labellings of all binary trees with three internal
nodes are

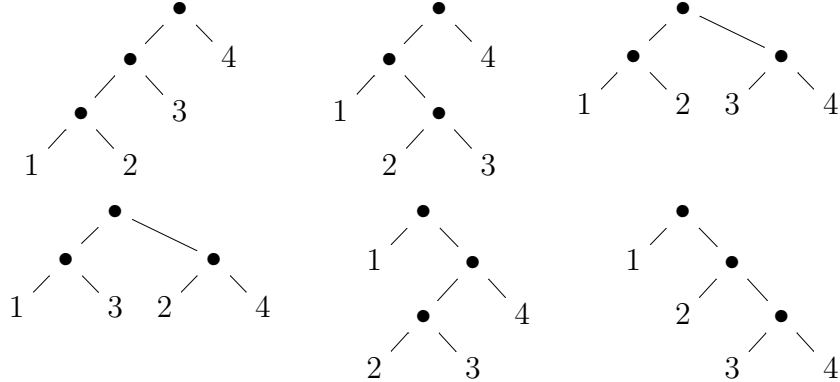
\newdimen\vcadre\vcadre=0.15cm 
\newdimen\hcadre\hcadre=0.15cm 
\setlength\unitlength{1.5mm}
\begin{equation*}
\begin{split}
{\xymatrix@C=2mm@R=2mm{
 *{} & *{} & *{} & {\bullet}\ar@{-}[dl]\ar@{-}[dr]  \\
 *{} & *{} & {\bullet}\ar@{-}[dl]\ar@{-}[dr] & *{} & 4 \\
 *{} & {\bullet}\ar@{-}[dl]\ar@{-}[dr] & *{} & 3 \\
 {1} & *{} & {2} 
      }}
\qquad
{\xymatrix@C=2mm@R=2mm{
 *{} & *{} & *{} & {\bullet}\ar@{-}[dl]\ar@{-}[dr]  \\
 *{} & *{} & {\bullet}\ar@{-}[dl]\ar@{-}[dr] & *{} & 4 \\
 *{} & 1   & *{} & {\bullet}\ar@{-}[dl]\ar@{-}[dr] \\
 *{} & *{} & {2} & *{} & 3 
      }}
\qquad
{\xymatrix@C=2mm@R=2mm{
 *{} & *{} & {\bullet}\ar@{-}[dl]\ar@{-}[drr]  \\
 *{} & {\bullet}\ar@{-}[dl]\ar@{-}[dr] & *{}
     & *{} & {\bullet}\ar@{-}[dl]\ar@{-}[dr] \\
  1  & *{} & 2 & 3 & *{} & 4
      }}
\\
{\xymatrix@C=2mm@R=2mm{
 *{} & *{} & {\bullet}\ar@{-}[dl]\ar@{-}[drr]  \\
 *{} & {\bullet}\ar@{-}[dl]\ar@{-}[dr] & *{}
     & *{} & {\bullet}\ar@{-}[dl]\ar@{-}[dr] \\
  1  & *{} & 3 & 2 & *{} & 4
      }}
\qquad
{\xymatrix@C=2mm@R=2mm{
 *{} & *{} & {\bullet}\ar@{-}[dl]\ar@{-}[dr]  \\
 *{} & 1 & *{} & {\bullet}\ar@{-}[dl]\ar@{-}[dr]  \\
 *{} & *{} & {\bullet}\ar@{-}[dl]\ar@{-}[dr] & *{} & 4 \\
 *{} & 2 & *{} & 3
      }}
\qquad
{\xymatrix@C=2mm@R=2mm{
 *{} & *{} & {\bullet}\ar@{-}[dl]\ar@{-}[dr]  \\
 *{} & 1   & *{} & {\bullet}\ar@{-}[dl]\ar@{-}[dr]  \\
 *{} & *{} & 2   & *{} & {\bullet}\ar@{-}[dl]\ar@{-}[dr] \\
 *{} & *{} & *{} & 3 & *{} & 4
      }}
\end{split}
\end{equation*}
The corresponding Lie elements are
\begin{equation}
\begin{split}
[[[1,2],3],4], \quad [[1,[2,3]],4] \quad [[1,2],[3,4]], \\
[[1,3],[2,4]], \quad [1,[[2,3],4]] \quad [1,[2,[3,4]]],
\end{split}
\end{equation}
the smallest example of a tree with two admissible labellings being
 \begin{equation}
[[1,2],[3,4]],\quad [[1,3],[2,4]].
\end{equation}

Note in particular that there is only one admissible labelling of the
left-comb tree and that this element is equal to the Dynkin
element~\eqref{eq:dyn}.

\medskip
It is known~\cite{ST} that the number of such labelled trees is indeed
$(n-1)!$ and that they are linearly independent.
More precisely, one can count the number of admissible labellings of a given
tree.

Recall that the \emph{decreasing tree} $\DT(w)$ of a word without repeated
letters $w=unv$ is the binary tree with root $n=\max(w)$, left subtree
$\DT(u)$ and right subtree $\DT(v)$. 

The number of permutations of $\SG_n$ with a given decreasing tree $t$ is
given by the \emph{hook-length formula}
\begin{equation}
N(t) = n!\prod_{v\in t}\frac1{h(v)}
\end{equation}   
where the hook-length $h(v)$ of a node $v$ is the number of nodes of its
subtree.

One can show (see Appendix) that the admissible labellings of a complete
binary tree $T$ are in bijection with the permutations whose decreasing tree
has shape $t$, the binary tree consisting of the internal nodes of $T$. 

\subsection{The convolution Lie algebra and its Catalan subalgebra}

The following easy proposition appears to have remained unnoticed.

\begin{proposition}
\noindent (i) The direct sum
\begin{equation}
\Lie = \bigoplus_{n\ge 1}\Lie(n)
\end{equation}
is a Lie algebra for the convolution bracket 
\begin{equation}
[\alpha,\beta]_\star = \alpha\star\beta - \beta\star\alpha
\end{equation}
of $\K\SG$.

(ii) Representing its elements as complete binary trees with labelled leaves
$T(\sigma)$, one has
\begin{equation}
[T_1(\alpha),T_2(\beta)]_\star
 = \sum_{\gamma=\gamma_1.\gamma_2\in\alpha*\beta}[T_1(\gamma_1),T_2(\gamma_2)]
\end{equation}
where $\gamma_1$ is the prefix of $\gamma$ of size equal to the size of
$\alpha$.
Note that in the r.h.s., the bracket is taken w.r.t. concatenation.

(iii) $\Lie$ is (strictly) contained in the primitive Lie algebra of $\K\SG$.
\end{proposition}

For example
\begin{equation}
\begin{split}
[[1,2],1]_\star
  &= [1,2]\star 1-1\star[1,2] \\
  &= [1,2].3 + [1,3].2 + [2,3].1 - 1.[2,3] - 2.[1,3] - 3.[1,2] \\
  &=  [[1,2],3]+[[1,3],2]+[[2,3],1]. \\
\end{split}
\end{equation}

\Proof 
Points (i) and (ii) are consequences of the following identity.
Let $\alpha$ and $\beta$ be two permutations.
Then,
\begin{equation}
[\alpha,\beta]_\star = \alpha\star\beta - \beta\star\alpha =
\sum_{\gamma=\gamma_1.\gamma_2\in\alpha*\beta}
    ({\gamma_1\cdot\gamma_2} - {\gamma_2\cdot\gamma_1}),
\end{equation}
which is immediate from the expression of the convolution product
\eqref{eq:star}.

Point (iii) is a consequence of Ree's criterion (cf. \cite{Re}): in the free
associative algebra, the free Lie algebra is the orthogonal of the space of
proper shuffles. But in $\K\SG$, an element is primitive iff it is othogonal
to special shuffles $u\shuffle v$, where $u$ and $v$ are permutations of
consecutive intervals $[1,k]$ and $[k+1,n]$.
\qed

It is known that the primitive Lie algebra of $\K\SG$ is free \cite{Foi1}.
Therefore, $\Lie$, the convolution Lie algebra, which is a Lie subalgebra of
the previous one, is free as well (see Section 2.2 of~\cite{Re}).

Moreover, the PBW basis of $\Lie$ allows one to define an interesting Lie
subalgebra implicitly defined in~\cite{BS}:

\begin{theorem}[\cite{BS}]
\label{th:Cat}
For a complete binary tree $T$, let
\begin{equation}
\c_T = \sum_{\sigma\ {\rm admissible}}T(\sigma).
\end{equation}
Then, the $\c_T$ span a Lie subalgebra $\CC$ of $\Lie$.
\end{theorem}

Our proof of this fact relies upon the quadri-algebra structure,
introduced in Section~\ref{sec:quadri}.
Actually, $\CC$ is also a preLie algebra as already observed
in~\cite{BS}.

Let $T$ and $T'$ be two binary trees.
We shall denote by $T\wedge T'$ the tree whose left (resp. right) subtree is
$T$ (resp. $T'$).
This grafting operation is understood as bilinear.
Let $\triangleright$ be the preLie product as in Eq. \eqref{eq:prelie}.

\begin{proposition}
\label{prop:prel}
$\CC$ is stable for the preLie product $\triangleright$.

Writing $T$ for $\c_T$, the preLie product satifies the recursion 
\begin{equation}
\begin{split}
T_1 \droit (T_2\wedge T_3) =&
    (T_1 \wedge (T_2\wedge T_3)) 
  + ((T_1\droit T_2) \wedge T_3)
  + (T_2\wedge (T_1\droit T_3)) \\
  &- ((T_1\wedge T_2) \wedge T_3)
  - ((T_2\wedge T_1) \wedge T_3),
\end{split}
\end{equation}
or, pictorially,
\begin{equation}
\label{eq-prelie-nus}
\begin{split}
T_1 \droit \arbgx{T_2}{T_3} =&
\arbgxb{T_1}{T_2}{T_3} +
\arbgx{T_1\droit T_2}{T_3} +
\arbgx{T_2}{T_1\droit T_3} \\
&- \arbgxc{T_1}{T_2}{T_3}
 - \arbgxc{T_2}{T_1}{T_3}\,.
\end{split}
\end{equation}
\end{proposition}

\Proof
Let us define a new law $\droit'$ as
$T\droit' T'=T\droit T' - T \wedge T'$.
Then the statement rewrites as
\begin{equation}
\label{eq-prelie-nusb}
\begin{split}
T_1 \droit' \arbgx{T_2}{T_3} =&
\arbgx{T_1\droit' T_2}{T_3} +
\arbgx{T_2}{T_1\droit' T_3}\\
& + \arbgxb{T_2}{T_1}{T_3} 
  - \arbgxc{T_2}{T_1}{T_3}\,.
\end{split}
\end{equation}

We have seen that the product of $\FQSym$ can be split into the sum of the
four quadri-algebra products, so that the Lie bracket consists of eight
different terms. Let us group the terms as follows:

\begin{equation}
\begin{split}
S_{E-W}(T_1,T_2) = T_1 \searrow T_2 - T_2\swarrow T_1 \\
S_{W-E}(T_1,T_2) = T_1 \swarrow T_2 - T_2\searrow T_1 \\
N_{W-E}(T_1,T_2) = T_1 \nwarrow T_2 - T_2\nearrow T_1 \\
N_{E-W}(T_1,T_2) = T_1 \nearrow T_2 - T_2\nwarrow T_1 \\
\end{split}
\end{equation}

Then, 
\begin{equation}
S_{E-W}(T_1,T_2) = \arbgx{T_1}{T_2},\ 
S_{W-E}(T_1,T_2) = -\arbgx{T_2}{T_1},
\end{equation}

\begin{equation}
\begin{split}
N_{W-E}\left(\arbgx{T_1}{T_2},T_3\right)
&=
   \arbgx{S_{E-W}(T_1,T_3)}{T_2}
 + \arbgx{N_{W-E}(T_1,T_3)}{T_2} \\
&+ \arbgx{T_1}{S_{W-E}(T_2,T_3)}
   \arbgx{T_1}{N_{W-E}(T_2,T_3)},
\end{split}
\end{equation}
and
\begin{equation}
N_{E-W}(T_1,T_2) = - N_{W-E}(T_2,T_1).
\end{equation}

All these formulas are easily obtained by following carefully what happens to
the smallest and largest values in the different products.

It follows that
\begin{equation}
T_1 \droit T_2
 = T_1 \succ T_2 - T_2 \prec T_1 = S_{E-W}(T_1,T_2) + N_{E-W}(T_1,T_2),
\end{equation}
whence the formula for the pre-Lie product of naked trees
\eqref{eq-prelie-nus}.
\qed

In \cite{BS}, a pre-Lie algebra structure is defined on the linear span of
(abstract) complete binary trees, by means of a combinatorial formula
defined in terms of graftings. We are now in a position to see that
our pre-Lie structure on $\CC$ coincides with that of \cite{BS}.

\begin{proposition}
The Catalan preLie algebra $\CC$ is isomorphic to the preLie algebra
${\mathcal T}_{\rm pb}$ defined in \cite[Sec. 3.2]{BS}.
\end{proposition}

\Proof
In the setting of \cite{BS}, computing $T\droit' T'$ amounts to taking the sum of all
(naked) trees obtained by creating a new internal node $n$ in the middle of
any right branch and gluing $T$ as left subtree of $n$, the right subtree of
$n$ being what was below $n$ in the beginning, minus the similar sum where one
creates a new internal node in the middle of any left branch and gluing $T$ as
right subtree.
Indeed, Formula~\eqref{eq-prelie-nusb} consists in the creation of two nodes
on the branches connected to the root (last two terms) and in the induction on
the left and right subtrees (first two terms).
\qed

\medskip 
Recall that the primitive Lie algebra of $\NCSF$ is freely generated by the
$\Psi_n$.
Since $\Psi_n = (\cdots(\G_1\triangleright \G_1)\triangleright \cdots)
                 \triangleright \G_1\in{\CC}_n$, we have:

\begin{proposition}
The primitive Lie algebra of $\Sym$ is contained in the Catalan Lie algebra
${\CC}$.
\end{proposition}

It is therefore of interest to investigate the expansion of the various Lie
idempotents of the descent algebra on the basis $\c_T$. The Dynkin elements
are obviously the left and right comb trees. The expansion of the Solomon
idempotent is obtained in \cite{BS}. In the sequel, we shall obtain the
following expansion for the Catalan idempotents of \cite{MNT}:

\begin{theorem}
\label{catpbw}
For an element $T(\sigma)$ of the PBW basis of $\Lie(n)$, denote by $t$ the
binary tree consisting of its internal nodes, and let $r(t)$ and $l(t)$ be
respectively the number of right edges and left edges of $t$. Then,
\begin{equation}
D_{a,b}^n = \sum_{T(\sigma)}a^{r(t)}b^{l(t)} T(\sigma),
\end{equation}
where the sum is over all admissible labelled trees.
\end{theorem}

In particular, the sum of the PBW basis is a Lie idempotent of the descent
algebra.

The proof of this result involves the fine structure of $\FQSym$, as a
dendriform algebra, and also as a quadri-algebra. As we shall see, it is not
obvious to determine whether an element of ${\CC}$ belongs to the
descent algebra. This question motivates the developments of
Section~\ref{sec:subalg}.

\section{The Catalan idempotents in the PBW basis}
\label{sec:quadri}

This section is devoted to the proof of Theorem~\ref{catpbw}.


\subsection{A functional equation}

For a complete binary tree $T$, let $T_\nesw$ be the evaluation of $T$ with
$\G_1$ in all the leaves, and the operation $\nesw$ in the internal nodes.

\begin{proposition}\label{prop:swne}
The operation $\nesw$ is magmatic, and
\begin{equation}
T_\nesw = \sum_{\sigma\ {\rm admissible}}T(\sigma) = \c_T
\end{equation}
is the sum of all trees of shape $T$ in the PBW basis.
\end{proposition}

\Proof
This follows from Prop. \ref{prop:labels} with $B=\nesw$.
\qed

It follows from the previous considerations that the formal sum $X$ of the PBW
basis satisfies the functional equation
\begin{equation}
X = \G_1 + B(X,X),\quad \text{where $B(X,Y)=X\nesw Y$},
\end{equation}
and setting $X=\G_1+X_+$, we can introduce parameters to count left and right
internal branches: the sum $X$ of the expressions for the Catalan idempotents
proposed in Theorem~\ref{catpbw} is the unique solution of
\begin{equation}\label{eq:funeq}
X_+ = B(\G_1+aX_+,\G_1+bX_+).
\end{equation}

From now on, we shall write $\sigma$ for $\G_\sigma$, and set, for $n\geq 2$,
\begin{align}
X_n&=\sum_{\sigma \in \mathfrak{S}_n} c_\sigma \sigma.
\end{align}
Let $\sigma\in \SG_n$, with $n\geq 2$. Let us first consider the case where
$n$ precedes $1$ in $\sigma$.
Since $X = \G_1 + X\nearrow X - X\swarrow X$, the products contributing to
$c_\sigma$ correspond to products $\tau\swarrow\mu$ containing $\sigma$ which
amounts to to considering all ways of writing $\sigma$ as $u\cdot v$ with
$1\in v$ and $n\in u$. The other case for $\sigma$ is similar so that
$c_\sigma$ is inductively
\begin{equation}
\label{eq-rec-csig}
c_\sigma =
\sum_{\gf{\sigma=u.v}{1\in u, n\in v}} a_{|u|}b_{|v|} c_{\std(u)}c_{\std(v)}
- \sum_{\gf{\sigma=u.v}{n\in u, 1\in v}} b_{|u|}a_{|v|} c_{\std(u)}c_{\std(v)}.
\end{equation}

Note that only one sum contributes to a nonzero coefficient for $c_\sigma$.

\begin{lemma}
For  $\sigma\in \SG_n$, define
$\overline{\sigma}=\sigma\circ (n,\ldots,1)=(\sigma_n,\ldots,\sigma_1)$.
Then:
\begin{align*}
c_\sigma&=(-1)^{n-1}c_{\overline{\sigma}}.
\end{align*} \end{lemma}

\begin{proof}
This follows by induction from Eq. ~\eqref{eq-rec-csig}. Indeed,
\begin{equation}
\begin{split}
c_{\overline\sigma} &=
\sum_{\gf{\overline\sigma=u.v}{1\in u, n\in v}} a_{|u|}b_{|v|} c_{\std(u)}c_{\std(v)}
- \sum_{\gf{\overline\sigma=u.v}{n\in u, 1\in v}} b_{|u|}a_{|v|}
  c_{\std(u)}c_{\std(v)} \\
&= - \sum_{\gf{\sigma=u.v}{1\in u, n\in v}} a_{|u|}b_{|v|}
c_{\std(\overline u)}c_{\std(\overline v)}
+ \sum_{\gf{\overline\sigma=u.v}{n\in u, 1\in v}} b_{|u|}a_{|v|}
  c_{\std(\overline u)}c_{\std(\overline v)}.
\end{split}
\end{equation}
Now, $|\sigma|$ is equal to $|u|+|v|$ hence of the same parity so that we recover
$c_\sigma$ with the correct sign.
\end{proof}

The {\it Narayana polynomials} $N_n(a,b):=\ca_{n-1}(a,b)$
satisfy the recurrence
\begin{equation}
N_n=\begin{cases}
1\mbox{ if }n\leq 2,\\
\displaystyle (a+b)N_{n-1}+ab\sum_{k=2}^{n-2} N_k  N_{n-k}\mbox{ if }n\geq 3.
\end{cases}
\end{equation}


Let $\sigma\in \SG_n$. We define the maximal runs of $\sigma$ as the integers
$i_0=1<i_1<\ldots<i_p<n=i_{p+1}$ satisfying the conditions
\begin{itemize}
\item For any $0\leq k\leq p$, $\sigma_{{\mid \{i_k,\ldots i_{k+1}\}}}$ is
monotonous.
\item For any $0\leq k\leq p-1$, $\sigma_{{\mid \{i_k,\ldots i_{k+2}\}}}$ is
not monotonous.
\end{itemize}

\subsection{Proof of Theorem \ref{catpbw}.}

We are now in a position to compute $c_\sigma$ by induction. Thanks to
Eq.~\eqref{eq-rec-csig}, we know that $c_\sigma$ is obtained by
deconcatenating $\sigma$ at all positions between $1$ and $n$. The main
ingredient of the proof consists in observing that summing over positions
inside a given run gives a simple inductive formula.

\begin{proposition}
Let $\sigma\in \SG_n$. Let $(i_0,\dots,i_{p+1})$ be its sequence of runs and,
for $0\leq k\leq p$, put:
\begin{align*}
m_k&=\begin{cases}
a\mbox{ if $\sigma_{\mid \{i_k,\ldots i_{k+1}\}}$ is increasing},\\
b\mbox{ if $\sigma_{\mid \{i_k,\ldots i_{k+1}\}}$ is decreasing}.
\end{cases}
\end{align*}
Let $\des(\sigma)$ denote the number of descents of $\sigma$. Then:
\begin{align}
c_{\sigma}
&=(-1)^{\des(\sigma)}\prod_{k=0}^{p-1} m_k \prod_{k=0}^p N_{i_{k+1}-i_k+1}.
\end{align}
\end{proposition}

\begin{proof}
We proceed by induction on $n$. This is obvious if $n=1$.

Let us first assume that $\sigma^{-1}_1<\sigma^{-1}_n$.
There exists $0\leq \alpha <\beta \leq k+1$ such that $\sigma^{-1}_1=i_\alpha$
and $\sigma^{-1}_n=i_\beta$. Hence:
\begin{align*}
c_\sigma&=\sum_{k=\alpha}^{\beta-1} \sum_{\ell=0}^{i_{k+1}-i_k-1}
a_{i_1+\ldots+i_k+\ell}b_{n-i_1-\ldots-i_k-\ell}
c_{\std(\sigma_1,\ldots,\sigma_{i_1+\ldots+i_k+\ell})}
c_{\std(\sigma_{i_1+\ldots+i_k+\ell+1},\ldots,\sigma_n)}\\
&=(-1)^{\des(\sigma)}\sum_{k=\alpha}^{\beta-1}
 \sum_{\ell=0}^{i_{k+1}-i_k-1}
\frac{M}{m_k}
N_{i_1-i_0+1}\ldots N_{i_k-i_{k-1}+1}N_{i_{k+2}-i_{k+1}+1}\ldots
N_{i_{p+1}-i_p+1}\varepsilon_k \mu_{k,\ell},
\end{align*}
with
\begin{equation}
M =\prod_{i=0}^{p-1}m_i,
\end{equation}
\begin{equation}
\varepsilon_k=\begin{cases}
+1\mbox{ if $\sigma_{\mid \{i_k,\ldots i_{k+1}\}}$ is increasing},\\
-1\mbox{ if $\sigma_{\mid \{i_k,\ldots i_{k+1}\}}$ is decreasing};\\
\end{cases}
\end{equation}

\begin{equation}
\mu_{k,\ell} = \begin{cases}
m_k N_{i_{k+1}-i_k+1}\mbox{ if }\ell=0,\\
m_{k-1}N_{i_{k+1}-i_k+1}\mbox{ if }\ell=i_{k+1}-i_k-1,\\
m_{k-1}m_k N_{\ell+1}N_{i_{k+1}-i_k-\ell} \mbox{ if }1\leq \ell\leq i_{k+1}-i_k-2.
\end{cases}
\end{equation}

As $\{m_{k-1},m_k\}=\{a,b\}$, summing over $\ell$ yields
\begin{equation}
\sum_{\ell=0}^{i_{k+1}-i_k-1}\mu_{k,\ell}=N_{i_{k+1}-i_k+1},
\end{equation}
and we  obtain
\begin{equation}
c_{\sigma}=(-1)^{\des(\sigma)}\prod_{k=0}^{p-1} m_k \prod_{k=0}^p N_{i_{k+1}-i_k+1} \sum_{k=\alpha}^{\beta-1} \varepsilon_k.
\end{equation}

As $i_\alpha=\sigma^{-1}_1$,
$\sigma_{{\mid\{i_\alpha,\ldots,i_{\alpha+1}\}}}$ is increasing, so
$\varepsilon_k=(-1)^{k-\alpha}$.
As $i_{\beta}=\sigma^{-1}_n$,
$\sigma_{{\mid \{i_{\beta-1},\ldots,i_\beta\}}}$ is increasing, so
$\beta-\alpha-1$ is even.
Hence,
\begin{equation}
\sum_{k=\alpha}^{\beta-1} \varepsilon_k
=\sum_{k=0}^{\beta-\alpha-1}(-1)^k=\frac{1-(-1)^{\beta-\alpha}}{2}=1,
\end{equation}
which finally gives the announced result for $c_\sigma$.

\medskip
Let us now assume that $\sigma^{-1}_n<\sigma^{-1}_1$. Applying the previous
result to $\overline{\sigma}$, we obtain:
\begin{equation}
c_{\sigma}
=(-1)^{n-1-\des(\overline{\sigma})}
  \prod_{k=0}^{p-1} m'_k \prod_{k=0}^p N_{i_{k+1}-i_k+1}.
\end{equation}
If $p$ is even, then
\begin{equation}
\prod_{k=0}^{p-1} m'_k=\prod_{k=0}^{p-1} m_k=(ab)^{\frac{p}{2}}.
\end{equation}
If $p$ is odd, then
\begin{align}
\prod_{k=0}^{p-1} m_k&=\begin{cases}
(ab)^{\frac{p-1}{2}}a\mbox{ if $\sigma_{{\mid \{i_0,\ldots,i_1\}}}$ is increasing},\\
(ab)^{\frac{p-1}{2}}b\mbox{ if $\sigma_{{\mid \{i_0,\ldots,i_1\}}}$ is decreasing},\\
\end{cases}\\
\prod_{k=0}^{p-1} m'_k&=\begin{cases}
(ab)^{\frac{p-1}{2}}a\mbox{ if $\sigma_{{\mid \{i_p,\ldots,i_{p+1}\}}}$ is decreasing},\\
(ab)^{\frac{p-1}{2}}b\mbox{ if $\sigma_{{\mid \{i_p,\ldots,i_{p+1}\}}}$ is increasing}.
\end{cases}
\end{align}
As $p$ is odd, $\sigma_{{\mid \{i_p,\ldots,i_{p+1}\}}}$ is decreasing if, and
only if, $\sigma_{{\mid \{i_0,\ldots,i_1\}}}$ is increasing.
Therefore, in all cases, 
\begin{equation}
\prod_{k=0}^{p-1} m'_k=\prod_{k=0}^{p-1} m_k,
\end{equation}
and finally the result is proved for any $\sigma$.
\end{proof}

\begin{corollary}
If the descent sets of $\sigma$ and $\tau$ are equal, then $c_\sigma=c_\tau$.
\end{corollary}

Comparing the explicit expressions for the coefficients, this proves
Theorem~\ref{catpbw}.

\section{A new basis of $\CC$}

\subsection{From binary trees to plane trees}

In this section, it will be convenient to label the basis elements of
$\CC$ by plane rooted trees instead of binary trees. 

The Butcher product $\T_1 \Rsh \T_2$ of two plane rooted trees $\T_1$, $\T_2$
is obtained by grafting $\T_1$ on the root of $\T_2$, on the left.
Then one defines the usual bijection from plane trees to binary trees as
the Knuth rotation $K$ recusively defined by $K(\bullet)=\epsilon$ (the
empty binary tree) and for any plane rooted trees $\T$ and $\T'$,
\begin{equation}
K(\T\Rsh \T')=K(\T)\wedge K(\T').
\end{equation}
This correspondence is illustrated on Figures~\ref{tam3} and~\ref{tam4} at the
end of the paper, representing the Tamari order for $n=3,4$.

\medskip
We shall make use of some typographical distinctions to help understand which
object is of what type: if $T$ is a complete binary tree, we denote by $t$ the
binary tree obtained by removing its leaves, and by $\T$ the corresponding
plane tree.

\medskip
The cover relation of the Tamari order on plane trees is described as follows:
starting from a tree $\T$ and a vertex $x$ that is neither its root or a leaf,
the trees $\T'>\T$ covering $\T$ are obtained by cutting off the leftmost
subtree of $x$ and grafting it back on the left of the parent of $x$.

\medskip
Now, in the $\P$ basis of $\PBT$ ({\it cf. }\cite{HNT} for background and notations), 
the product $\P_{t_1}\P_{t_2}$ is an
interval of the Tamari order, which is described on plane trees as follows.
Let us define $\T_1=B(\U_1\cdots \U_r)$ where $\U_1\dots\U_r$ are the subtrees
of $\T_1$ and let $b_1,\dots,b_n$ be the vertices of the leftmost branch of
$\T_2$, $b_1$ being its root and $b_n$ its leftmost leaf.

Then, any nondecreasing sequence
$0=i_0\leq i_1\leq i_2\leq\dots\leq i_n=r$ of $n$ indices
corresponds to one term of the product $\P_{\T_1}\P_{\T_2}$:
define the plane forest $F_1^{(k)}=(\U_{i_{k-1}+1},\dots,\U_{i_{k}})$ for all
$k\in[1,n]$.
The corresponding term of the product (in the basis $\P$) is the tree obtained
by grafting all the elements of the forest $F_1^{(i)}$, respecting their order,
on the vertex $b_i$ and to the left of $b_{i+1}$, see Figure~\ref{prodt1t2}.

\vskip1cm
\begin{figure}[ht]
\centerline{
\newdimen\vcadre\vcadre=0.01cm 
\newdimen\hcadre\hcadre=0.01cm 
\xymatrix@R=0.1cm@C=2mm{
 && {\GrTeXBox{b_1}}\arx1[dl]\arx1[dr]\arx1[drrr]\\
 & {\GrTeXBox{\U_1\dots \U_{i_1}}}  & *{}
   & {\GrTeXBox{b_2}}\arx1[dl]\arx1[dr]\arx1[drrr] && {\GrTeXBox{\dots}} \\
 && {\GrTeXBox{\U_{i_1+1}\dots \U_{i_2}}}
   && {\GrTeXBox{b_3}}\arx1[dl]\ar@{.}[dr]\arx1[drrr]
   && {\GrTeXBox{\dots}} \\
 &&& {\GrTeXBox{\U_{i_2+1}\dots \U_{i_3}}}
   && {\GrTeXBox{b_n}}\arx1[dl]
   && {\GrTeXBox{\dots}} \\
 &&&& {\GrTeXBox{\U_{i_{n-1}+1}\dots \U_{i_n}}} \\
}
}
\caption{
\label{prodt1t2}
A generic element of the product $\P_{\T_1} \P_{\T_2}$, where
$\U_1,\dots,\U_{i_n}=\U_r$ are, in this order, the children of the root of
$\T_1$ and the $b_i$ are the vertices of the left branch of $\T_2$.}
\end{figure}
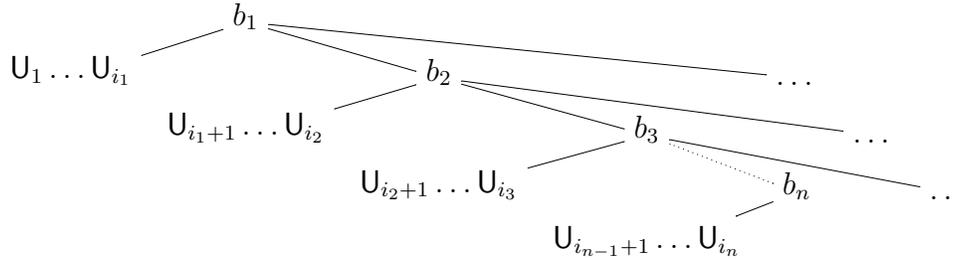

\subsection{A new basis of $\CC$}

Define a new basis $\x_t$ of $\CC_n$ (indexed by incomplete
binary trees of size $n-1$) by the condition
\begin{equation}\label{eq:basex} 
\c_t = \sum_{u\le t}\x_u
  \quad \Leftrightarrow\quad
\x_t = \sum_{u\le t}\mu(u,t)\c_u
\end{equation}
where $\le$ is the Tamari order, and $\mu$ its Moebius funtion.

For example,
\begin{align}
\x_{\Btabc} &= \c_{\Btabc}\\
\x_{\Btbac} &= \c_{\Btbac} - \c_{\Btabc}\\
\x_{\Btcab} &= \c_{\Btcab} - \c_{\Btbac}\\
\x_{\Btacb} &= \c_{\Btacb} - \c_{\Btabc}\\
\x_{\Btcba} &= \c_{\Btcba} - \c_{\Btacb} - \c_{\Btcab} + \c_{\Btabc}.
\end{align}

For a binary tree $t$ of size $n-1$, denote by $\T=K^{-1}(t)$ the plane
tree of size $n$ corresponding to $t$, and set $X_\T= \x_t$ and $C_\T=\c_t$.
Then on plane trees, the previous equations read as (see Figure~\ref{tam3}
for the bijection between incomplete binary trees of size $3$ and plane trees
of size $4$)

\begin{align}
X_{\arbtge} &= C_{\arbtge}                \\
X_{\arbtgc} &= C_{\arbtgc} - C_{\arbtge}  \\
X_{\arbtgd} &= C_{\arbtgd} - C_{\arbtgc}  \\
X_{\arbtgb} &= C_{\arbtgb} - C_{\arbtge}  \\
X_{\arbtga} &= C_{\arbtga} - C_{\arbtgd} - C_{\arbtgb} + C_{\arbtge}
\end{align}

\begin{theorem}
\label{th:xprel}
The preLie product in the $X$ basis is given by
\begin{equation}
\label{prelie-x}
X_{\T_1}\triangleright X_ {\T_2} = \sum_{\T\in G(\T_1,\T_2)}X_\T 
\end{equation}
where $G(\T_1,\T_2)$ is the multiset of trees obtained by grafting $\T_1$ on all
nodes of $\T_2$ in all possible ways.
\end{theorem}

For example,
\begin{align}
X_{\arbuga}\triangleright X_{\arbuga}
  &= X_{\arbtgb} + X_{\arbtgd} + X_{\arbtge}\\
X_{\arbdgb} \triangleright X_{\arbdga}
 &=  X_{\arbcga} + X_{\arbcgb} + X_{\arbcgc} + X_{\arbcgd} + X_{\arbcge} 
\end{align}

To prove this Theorem, we shall show  that
Equations~\eqref{prelie-x} and~\eqref{eq-prelie-nus} are equivalent.
To this aim, we define a new product $\btr$ 
by the condition that it satisfies~\eqref{prelie-x} and then show that it also
satisfies~\eqref{eq-prelie-nus}.
This is done in the following section.

\subsection{A preLie structure on $\PBT$}

Relation~\eqref{eq:basex} between the bases $x$ and $\c$ is the same
as the one between the natural basis $\P_t$ of $\PBT$ and the multiplicative
basis $\H_t$ defined in \cite[Eq. (46)]{HNT}.
We can therefore define a linear map
\begin{align}
\psi: &\PBT_{n-1}\longrightarrow \Lie(n)\nonumber\\
      & \P_t \longmapsto \x_t\\
      & \H_t \longmapsto \c_t\nonumber,
\end{align}
and define a preLie product on $\PBT$ by requiring that
\begin{equation}
\psi(\H_{t_1}\triangleright \H_{t_2}) = \c_{t_1}\triangleright \c_{t_2}.
\end{equation}
Theorem~\ref{th:xprel} can then be derived from a
compatibility property of the usual product of $\PBT$ with this preLie product.
Indexing as above the bases of $\PBT$ by plane trees instead of binary trees,
we define a new product $\blacktriangleright$ by
\begin{equation}
\P_\T\blacktriangleright \P_{\T'} = \sum_{\T''\in G(\T,\T')} \P_{\T''}.
\end{equation}

\begin{note}{\rm
We shall use the facts that $\P_\T\,\btr\,\P_\bullet =\P_{B(\T)}$ (by
definition since $B$ amounts to add a root on a sequence of trees)
and that $\H_\T\blacktriangleright\H_\bullet =\H_{B(\T)}$ since the trees
smaller than $B(\T)$ in the Tamari order are the $B(\T')$ with $\T'\leq \T$.
}
\end{note}

\begin{proposition}
For $U,V,W\in\PBT$,
\begin{equation}
\label{eq:quasider}
U\blacktriangleright (VW)
  = (U\blacktriangleright V)W 
  + V(U\blacktriangleright W)
  - V(U\blacktriangleright\bullet)W,
\end{equation}
where $\bullet =\P_\bullet=\H_\bullet$.
\end{proposition}

The idea behind this formula is fairly simple: the left hand-side consists in
gluing $U$ either on $W$ or on $V$. The first two terms amount to doing
essentially that, except that an extra term appears that should not be there,
where $V$ is glued on $U$, hence the corrective term with a minus sign.

\Proof
Since \eqref{eq:quasider} is linear in $U,V,W$, we can assume that
\begin{equation}
U = \P_{\T_0}, \quad V = \P_{\T_1}, \quad W = \P_{\T_2}.
\end{equation}
In that case, let $b_1,\dots,b_n$ be as above the vertices the leftmost
branch of $\T_2$. 
Then the first product $\P_{\T_0} \blacktriangleright (\P_{\T_1} \P_{\T_2})$
is obtained by summing over three disjoint sets of grafting patterns that we
will denote by $(a)$, $(b)$ and $(c)$.

Set $(a)$ consists of the grafting patterns where $\T_0$ is grafted on a $b_i$
and to the left of $b_{i+1}$.
Set $(b)$ consists of those where $\T_0$ is grafted on a vertex belonging to
$\T_1$.
Set $(c)$ consists of all other elements.
We also denote by $(d)$ the set of trees obtained in the product
$\P_{\T_1} (\P_{\T_0} \blacktriangleright \bullet) \P_{\T_2}$.

Let us now consider the other two terms.
The term $(\P_{\T_0} \blacktriangleright \P_{\T_1}) \P_{\T_2}$ also gives rise to two
cases: either $\T_0$ is grafted on the root of $\T_1$ (set $(a')$), 
or not (set $(b')$).
Finally, the term $\P_{\T_1} (\P_{\T_0} \blacktriangleright \P_{\T_2})$
splits into two cases: either $\T_0$ is grafted on a $b_i$ to the left of
$b_{i+1}$ (set $(d')$) or $\T_0$ is grafted somewhere else (set $(c')$).

Let us show that the sets labelled with the same letters coincide.

First, consider an element $\T$ of $(a')$. From the product formula on plane
trees, $\T$ is a tree obtained by grafting $\T_0$ on the root of $\T_1$, and
then grafting all the children of the root of this new tree, respecting their
order, on the $b_i$s of $\T_2$, to the left of $b_{i+1}$. Whether we put the
children of $\T_1$ on the left of the $b_i$s and then $\T_0$ on the same spots
or put them all together is irrelevant, so that sets $(a)$ and $(a')$ coincide.

Let us now prove that $(b)=(b')$.
Let $\T$ be in $(b')$. It has been obtained by grafting $\T_0$ somewhere on a
subtree $\U_i$ of $\T_1$ and then by grafting all subtrees of $\T_1$ on the $b_i$s
as usual. This is the same as first grafting all subtrees of $\T_1$ on
the $b_i$s and then grafting $\T_0$ on a subtree that was initially a subtree
of $\T_1$.
So $(b)=(b')$.

The equality $(c)=(c')$ is easy: the operations of grafting the $\U_i$s on the
$b_i$s to the left of $b_{i+1}$ and then grafting $\T_0$ not on that spots
clearly commute, which means precisely that $(c)=(c')$.

Finally, let us prove that $(d)=(d')$.
Observe that the product $\P_{\T_0}\blacktriangleright \P_{\T_2}$ decomposes into
two cases, depending on whether $\T_0$ is grafted on a $b_i$ left of $b_{i+1}$
or not. If it is the case (from which case $(d')$ is derived by multiplying by
$\P_{\T_1}$ on the left), it corresponds to a term of the product
$(\P_{\T_0}\blacktriangleright\bullet) \P_{\T_2}$ (from which case $(d)$ is
derived) also by multiplying by $\P_{\T_1}$ on the left. Hence $(d)=(d')$.

This proves that
\begin{equation}
\P_{\T_0}\blacktriangleright (\P_{\T_1}\P_{\T_2})
 + \P_{\T_1}(\P_{\T_0}\blacktriangleright\bullet)\P_{\T_2}
= (\P_{\T_0}\blacktriangleright \P_{\T_1})\P_{\T_2}
 + \P_{\T_2}(\P_{\T_0}\blacktriangleright \P_{\T_2})  
\end{equation}
which is equivalent to \eqref{eq:quasider}.
\qed

\begin{lemma}
Let $\T$ be a tree of the form $\T=B(\T')$ and $\T_0$ be an arbitrary tree.
Then,
\begin{equation}
\label{eq:HT0Ti}
\H_{\T_0}\blacktriangleright \H_\T
= \H_{B(\T_0\T')} + \H_{B(\T'\T_0)}
 - \H_{B(\T_0\Rsh \T')} - \H_{B(\T'\Rsh \T_0)}
 + \H_{B(\T_0\blacktriangleright \T')}.
\end{equation}
\end{lemma}

\Proof
In the $\P$ basis,
\begin{equation}
\H_{\T_0}\blacktriangleright \H_\T=
\left(
\sum_{\T'_0\leq \T_0} \P_{\T'_0}\right) \blacktriangleright
\left( \sum_{\T''\leq \T'} \P_{B(\T'')}\right).
\end{equation}
Expanding the product, we find three kinds of terms:
either $\T'_0$ is grafted on the root of $B(\T'')$ to the left or to the right
of $\T''$ (two cases) or it is grafted below it.
This last case sums up to $\H_{B(\T_0\btr \T')}$.
The other two cases are identical up to exchanging the roles of $\T_0$ and
$\T'$.
Let us deal with the first case. All trees appearing by grafting a
$\T'_0\leq \T_0$ to the left of the root of a $B(\T'')$ with $\T''\leq \T'$,
are all smaller in the Tamari order than $B(\T_0\T')$. Conversely, all trees
smaller than $B(\T_0\T')$ belong to the first case, except those where $\T'_0$
was grafted below the root. But  these trees  are precisely the elements
smaller in the Tamari order than $B(\T_0\Rsh \T')$.
So the terms belonging to the first case  sum up to
$\H_{B(\T_0\T')} - \H_{B(\T_0\Rsh \T')}$.

Analogously, the terms of the second case sum up to
$\H_{B(\T'\T_0)} - \H_{B(\T'\Rsh \T_0)}$.
Adding the contributions of the three cases, we obtain
Formula~\eqref{eq:HT0Ti}.
\qed

\Proof[of Theorem~\ref{th:xprel}]
We are now is a position to prove Theorem~\ref{th:xprel} by showing
that $\btr$ and $\triangleright$ coincide.

We want to compute a generic product $\H_{\T_1} \btr \H_\T$.
Since the $\H$ basis in multiplicative, we can replace $\H_\T$ by
$\H_{\T_2}\H_{\T_3}$ where $\T_2$ has a single tree attached: $T_2=B(\T'_2)$.

Now let us apply Formula~\eqref{eq:quasider} on the $\H$ basis:
\begin{equation}
\H_{\T_1} \btr (\H_{\T_2} \H_{\T_3})
= (\H_{\T_1} \btr \H_{\T_2}) \H_{\T_3} 
  + \H_{\T_2} (\H_{\T_1}\btr \H_{\T_3})
  - \H_{ \T_2} (\H_{\T_1}\blacktriangleright\bullet) \H_{\T_3}.
\end{equation}
By~\eqref{eq:HT0Ti}, the first term can be rewritten as
\begin{equation}
\left(\H_{B(\T_1\T'_2)} + \H_{B(\T'_2\T_1)}
 - \H_{B(\T_1\Rsh \T'_2)} - \H_{B(\T'_2\Rsh \T_1)}
 + \H_{B(\T_1\blacktriangleright \T'_2)} \right) \H_{\T_3},
\end{equation}
and now its second term $\H_{B(\T'_2\T_1)}\H_{\T_3}$ cancels with
$\H_{\T_2} (\H_{\T_1}\blacktriangleright\bullet) \H_{\T_3}$, 
so that
\begin{equation}
\begin{split}
\H_{\T_1} \btr (\H_{\T_2} \H_{\T_3})
&= 
\H_{B(\T_1\T'_2)} \H_{\T_3}
 + \H_{B(\T_1\btr \T'_2)} \H_{\T_3}
 + \H_{\T_2} (\H_{\T_1}\btr \H_{\T_3})\\
&
 - \H_{B(\T_1\Rsh \T'_2)} \H_{\T_3}
 - \H_{B(\T'_2\Rsh \T_1)} \H_{\T_3},
\end{split}
\end{equation}
which is exactly \eqref{eq-prelie-nus} expressed on plane trees.
This proves that $\btr=\triangleright$ by induction on the sizes of the
trees.
\qed

\begin{note}{\rm
The preLie product in the $X$-basis coincides with that of the free brace
algebra on one generator, which is the linear span of all plane trees, endowed
with the brace product
\begin{equation}
\<\T_1\T_2\cdots \T_k, \T_{k+1}\> = \sum_{\T\in G(\T_1,\T_2,\ldots \T_k; \T_{k+1})}\T,
\end{equation}
where $G(\T_1,\T_2,\ldots \T_k; \T_{k+1})$ is the multiset of trees obtained by
grafting $\T_1,\ldots,\T_k$ in this order on all nodes of $\T_{k+1}$ in all
possible ways \cite{Foi3}.

This brace product does not coincide with the one induced by the dendriform
structure of $\FQSym$. Indeed, on can check that ${\CC}$ is not stable
for this one: $\<\<\bullet,\bullet\>\bullet,\bullet\>$ is not in $\Lie$.
However, both products induce the same preLie structure, as well as the
extended preLie products 
\begin{equation}\label{eq:og}
\{\T_1\T_2\cdots\T_r;\T\} = \sum_{\sigma\in\SG_r}\<\T_{\sigma(1)}\T_{\sigma(2)}\cdots\T_{\sigma(r)};\T\>
\end{equation}
defined by their symmetrized versions (the
Oudom-Guin construction \cite{OG}).
}
\end{note}


\subsection{The Catalan idempotent in the $X$-basis}

Let $\PL$ be the free preLie algebra on one generator, and let
$p_\tau$ be its Chapoton-Livernet basis \cite{CL}, indexed by rooted
trees. Recall that the preLie product in this basis is given by
\begin{equation}
p_{\tau_1}\triangleright p_{\tau_2} = \sum_{\tau\in g(\tau_1,\tau_2)}p_\tau,
\end{equation}
where $g(\tau_1,\tau_2)$ is the multiset of trees obtained by grafting the
root of $\tau_1$ on all nodes of $\tau_2$.

By definition, ${\CC}$ contains the free preLie algebra generated by
$X_\bullet=\G_1$.
As an easy consequence of Theorem \ref{th:xprel}, we can express its
Chapoton-Livernet basis in terms of the $X$-basis of ${\CC}$.
\begin{lemma}
For a rooted plane tree $\T$, denote by $\bar\T$ the underlying
non-plane rooted tree.
The map
\begin{equation}
\iota:\ p_\tau\longmapsto |{\rm Aut}(\tau)|\sum_{\bar \T=\tau}X_\T
\end{equation}
is an embedding of preLie algebras.
\end{lemma} 
\qed

In terms of the symmetrized brace product \eqref{eq:og},
\begin{equation}
p_{B(\tau_1\cdots\tau_k)}=\{p_{\tau_1}\cdots p_{\tau_k},X_\bullet\}.
\end{equation}
We shall therefore identify $\PL$ with its image under $\iota$.

As already mentioned, the primitive Lie algebra of $\NCSF$, being generated
by the $\Psi_n$, is a Lie subalgebra of $\PL$.
In \cite{Cha2}, Chapoton gives an expression of the Catalan idempotent
on the basis $p_\tau$: setting the unnecessary parameter $a$ to 1,
the coefficient of $p_\tau$ in $D^n_{1,b}$ is equal to the generating
polynomial of \emph{small closed flows} on $\tau$ by \emph{size}.

A small closed flow of size $k$ on $\tau$ is equivalent to the following data:
\begin{itemize}
\item a set $O$ of $k$ distinct vertices (outputs), which cannot be the root;
\item a multiset $I$ of $k$ vertices (inputs);
\item a perfect matching between $I$ and $O$ such that each pair $(i,o)$
determines a path from $i$ to $o$ directed towards the root.
\end{itemize}
Denoting by $\F'(\tau,k)$ the set of small flows of size $k$ on $\tau$
and by $d_\tau(b)$ the polynomial
\begin{equation}
d_\tau(b)= \sum_k\sum_{\phi\in\F'(\tau,k)}b^k,
\end{equation}
we have \cite{Cha2}
\begin{equation}
D^n_{1,b}=\sum_{|\tau|=n}d_\tau(b)p_\tau.
\end{equation}
The \emph{rate} of a vertex is the sum of the multiplicities of the inputs
minus the number of outputs in its subtree.

Define the canonical labelling $\can(\T)$ of the vertices of a plane tree $\T$
as the one which yields the identity permutation when the tree is traversed in
postfix order: the label of a node is the cardinality of its subtree plus the
number of nodes strictly to its left. For example,
{ \newcommand{\nodea}{\node[draw,circle] (a) {$12$}
;}\newcommand{\nodeb}{\node[draw,circle] (b) {$1$}
;}\newcommand{\nodec}{\node[draw,circle] (c) {$8$}
;}\newcommand{\noded}{\node[draw,circle] (d) {$6$}
;}\newcommand{\nodee}{\node[draw,circle] (e) {$5$}
;}\newcommand{\nodef}{\node[draw,circle] (f) {$2$}
;}\newcommand{\nodeg}{\node[draw,circle] (g) {$3$}
;}\newcommand{\nodeh}{\node[draw,circle] (h) {$4$}
;}\newcommand{\nodei}{\node[draw,circle] (i) {$7$}
;}\newcommand{\nodej}{\node[draw,circle] (j) {$10$}
;}\newcommand{\nodeba}{\node[draw,circle] (ba) {$9$}
;}\newcommand{\nodebb}{\node[draw,circle] (bb) {$11$}
;}

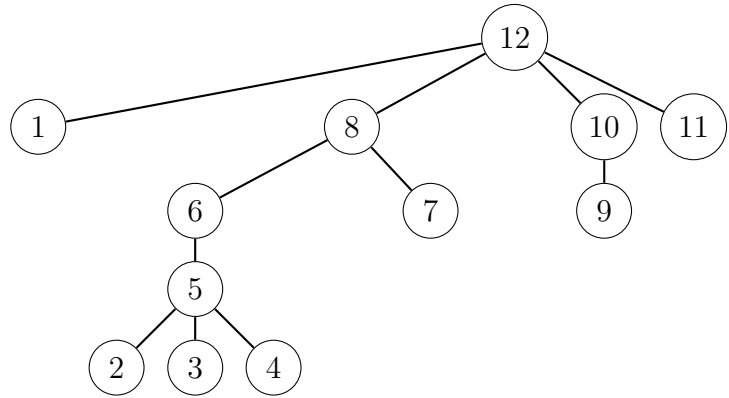
\begin{center}
\begin{tikzpicture}[auto]
\matrix[column sep=.3cm, row sep=.3cm,ampersand replacement=\&]{
         \&         \&         \&         \&         \&         \& \nodea  \&         \&         \\ 
 \nodeb  \&         \&         \&         \& \nodec  \&         \&         \& \nodej  \& \nodebb \\ 
         \&         \& \noded  \&         \&         \& \nodei  \&         \& \nodeba \&         \\ 
         \&         \& \nodee  \&         \&         \&         \&         \&         \&         \\ 
         \& \nodef  \& \nodeg  \& \nodeh  \&         \&         \&         \&         \&         \\
};

\path[thick, black] (e) edge (f) edge (g) edge (h)
	(d) edge (e)
	(c) edge (d) edge (i)
	(j) edge (ba)
	(a) edge (b) edge (c) edge (j) edge (bb);
\end{tikzpicture}
\end{center}
}

Extending the notion of flow to plane trees, we define the vector of a flow
$\phi$ on $\T$ as the tuple $V_\phi = (\phi(v_1),\ldots,\phi(v_n))$, where
$v_i$ are the vertices of $\T$ numbered by their canonical labelling.
The complete example of all small closed flows on plane trees on $4$ vertices can be found on
Figure~\ref{fig-flots}.

\begin{proposition}
If $\T'\ge \T$ in the Tamari order, the vector of any small closed flow
on $\T'$ is also the vector of a small closed flow on $\T$.
\end{proposition}

For example, on  Figure~\ref{fig-flots}, one can check that there
are five different flow vectors on plane trees of size 4 and that they all are flow vectors of the minimal
plane tree, the chain.

\Proof Let $v$ be the vector of a small closed flow on $\T'$.
If $\T'\ge \T$ are canonically labelled, and if $\T'$ covers $\T$,
the description of the Tamari order on plane trees shows that the subtree of a
node labelled $x$ in $\T'$, contains all the vertices of the subtree of the
node labelled $x$ in $\T$. 
Indeed, $\T'$ is obtained from $\T$ by cutting the subtree $\U$ of a vertex
$u$, and grafting it again on the left of its parent. The labels of the
vertices of $\U$, being equal to the cardinality of their subtree plus the
number of vertices to their left, remain unchanged by this operation, and the
other labels remain the same as well.
For example, the following tree covers the above one. The subtree of 5 has
been cut and grafted back on 8 without changing its labelling:
{ \newcommand{\nodea}{\node[draw,circle] (a) {$12$}
;}\newcommand{\nodeb}{\node[draw,circle] (b) {$1$}
;}\newcommand{\nodec}{\node[draw,circle] (c) {$8$}
;}\newcommand{\noded}{\node[draw,circle] (d) {$5$}
;}\newcommand{\nodee}{\node[draw,circle] (e) {$2$}
;}\newcommand{\nodef}{\node[draw,circle] (f) {$3$}
;}\newcommand{\nodeg}{\node[draw,circle] (g) {$4$}
;}\newcommand{\nodeh}{\node[draw,circle] (h) {$6$}
;}\newcommand{\nodei}{\node[draw,circle] (i) {$7$}
;}\newcommand{\nodej}{\node[draw,circle] (j) {$10$}
;}\newcommand{\nodeba}{\node[draw,circle] (ba) {$9$}
;}\newcommand{\nodebb}{\node[draw,circle] (bb) {$11$}
;}
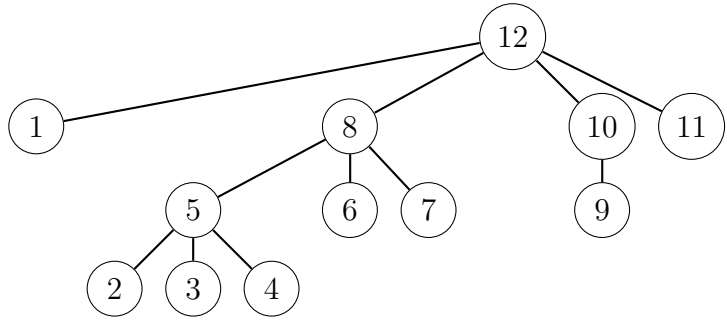
\begin{center}
\begin{tikzpicture}[auto]
\matrix[column sep=.3cm, row sep=.3cm,ampersand replacement=\&]{
         \&         \&         \&         \&         \&         \& \nodea  \&         \&         \\ 
 \nodeb  \&         \&         \&         \& \nodec  \&         \&         \& \nodej  \& \nodebb \\ 
         \&         \& \noded  \&         \& \nodeh  \& \nodei  \&         \& \nodeba \&         \\ 
         \& \nodee  \& \nodef  \& \nodeg  \&         \&         \&         \&         \&         \\
};

\path[thick, black] (d) edge (e) edge (f) edge (g)
	(c) edge (d) edge (h) edge (i)
	(j) edge (ba)
	(a) edge (b) edge (c) edge (j) edge (bb);
\end{tikzpicture}
\end{center}
}

Thus, the rate of $x$ in $\T$ is at least the rate of $x$ in $\T'$, so that it
is in particular non-negative.
Hence, $v$ is the vector of a flow on $T$, which is obviously small and
closed.
\qed

Conversely, if a perfect matching between $I$ and $O$ never go through a left
edge of a tree $\T$, then the same values on the same elements give rise to
another perfect matching on the tree $\T'$ (a covering tree of $\T$) obtained
by cutting this edge and gluing it on the node immediately above.

\medskip
Define the maximal small closed flow $\phi_0(\T)$ by choosing as outputs all
the internal vertices except the root, and the leftmost leaf of each subtree
as inputs with the maximal allowed value. This is a flow of maximal size,
whose vector is lexicographically maximal among all flows of $T$.

For example, the maximal flow on the first tree above is 
\begin{equation}
[0, 3, 0, 0, -1, -1, 0, -1, 1, -1, 0, 0],
\end{equation}
and the maximal flow of the second one is 
\begin{equation}
[0, 2, 0, 0, -1, 0, 0, -1, 1, -1, 0, 0].
\end{equation}
One easily checks that both flows are flows of the first tree but not of the
second one (the leaf $6$ has a $-1$ as value).

Let $k_0(\T)$ be the size of $\phi_0(\T)$.

\begin{proposition}
Given a plane tree $\T$, there is a bijection between all small closed flow on
$\T$ and the set of maximal flows on all trees $\T'\ge \T$.
The flow vector is preserved in the bijection.
\end{proposition}

\Proof
We have already seen that each flow vector of $\T'>\T$ is a flow vector of $\T$.
So in particular, each maximal flow vector of a $\T'>\T$ is a flow vector of
$\T$.

We need to prove that all maximal flow vectors are distinct and that each flow
vector of $\T$ is indeed a maximal flow vector of a $\T'>\T$.

First, the maximal flow vectors are distinct since one can rebuild a tree
given its maximal flow vector $v$: the nodes $i$ such that $v_i=-1$ are the internal
nodes, those with $v_i>0$ are the leftmost leaves of a non-root node, and
the nodes with $v_i=0$ are the other leaves and the root.
So, reading the flow vector from right to left, there cannot be any ambiguity
on where node $i$ should go, and since nodes are added from right to left, no
two flows can give rise to the same tree.

Finally, let $\phi$ be a flow  on a tree $\T$.
If it is its maximal flow the
the statement is a tautology, so we  assume that it is not maximal.
We will prove  that $\phi$ is a flow on a tree $\T'>\T$ and conclude by induction
on the maximal distance from $\T$ to the maximal element of the Tamari
lattice, the corolla.
Consider the maximal flow $\phi'$ of $\T$. There is a node that is nonzero in
$\phi'$ and that does not have the same value in $\phi$. If it is an internal node,
it has $0$ value in $\phi$ and since it is the leftmost child of its parent, we
can cut its connection to its father and graft it on its grand-parent, then
obtaining a flow on a tree $\T'>\T$. If there is no such internal node 
differing between $\phi$ and $\phi'$, there is a leaf that does not have its
maximal possible value in $\phi$. Then, since all the internal nodes above it
except the root have value $-1$, let $n$ be the bottom-most such node that has
another flow coming to it from below. Then cut the left child of $n$ and graft
it on the parent of $n$. In the resulting tree $\T'$, the flow is still valid.
Indeed, by definition of a flow, if all matchings between inputs and outputs
correspond to a path towards to the root, then it is valid. In our case, $n$
can be matched to this \emph{other source}, so that all elements matching to
the leaf are either above or below $n$ and hence remain valid matchings in
$\T'$.

So we have proven that any flow of a tree $\T$ is either its maximal flow or a
flow of a tree $\T'>\T$. Iterating this process with $\T'$, it will have
to stop since there are only a finite number of trees greater than a given
tree.
\qed

Note that $k_0(\T)$ is the number of non-root internal vertices so that
$k_0(\T)=i(\T)-1$, where $i(\T)$ is the number of internal vertices.
It is also the number of left egdes of the corresponding binary tree.
Thus, the polynomials $d_\T(b)$ enumerate Tamari intervals according to these
statistics:

\begin{proposition}
\begin{equation}
d_\T(b) := \sum_k\sum_{\phi\in \F'(\T,k)}b^k = \sum_{\T'\ge \T} b^{k_0(\T')}.
\end{equation}
\end{proposition}

\qed

As a consequence, we recover Theorem \ref{catpbw}:
\begin{corollary}
\begin{equation}
D^n_{1,b} = \sum_{|t|=n-1}b^{l(t)}b_t.
\end{equation}
\end{corollary}

Thus, we have now two quite different explanations of the curious fact that
the Catalan idempotent is the weighted sum of the PBW basis. The first one
relies upon the functional equation \eqref{eq:funeq}, which involves the quadri-algebra
structure of $\FQSym$, and the second one, which is derived from the preLie expansion given
by Chapoton in terms of small closed flows. It would be interesting to investigate
in the same manner the PBW expansions of the new Lie idempotents
introduced in \cite[Prop. 6.8 and Conj. 6.10]{Cha2}.

\section{Lie and pre-Lie subalgebras of $\CC$}
\label{sec:subalg}

\subsection{Left equivalence}

As we have seen on the case of the Catalan idempotent, it is not easy to
decide whether an element of $\Lie(n)$ expressed on the PBW basis belongs
to the descent algebra $\Sym_n$. A similar question  arises with $\PBT$.
For example, an old conjecture by \'Ecalle, as reformulated by Chapoton, states
that $\PBT\cap\Lie$ is the free preLie algebra generated by $\G_1=\P_\bullet$.

The equivalences classes to be defined below arose from the following question:
in a generic linear combination  $g= \sum_t \alpha_t\c_t$ of the basis $\c_t$ of $\CC$, which trees
must have the same coefficient if one requires that $g\in\PBT$ or $g\in\NCSF$?

It turns out that the sums over equivalence classes
span Lie subalgebras of $\Lie$,
which are conjectured to contain respectively $\PBT\cap\Lie$ (for the
L-classses) and proved to contain $\Sym\cap\Lie$ (for the LR-classes). In both cases, the
bracket of these subalgebras admits a remarkable combinatorial description.

\subsubsection{An exchange rule}

Consider the following transformation on a binary tree $T$: choose a vertex
$v$ that can either be the root of $T$ or the right child of its parent.
Let $T_1$ and $T_2$ be the subtrees depicted below
\begin{equation}
\xymatrix@R=0.1cm@C=1mm{
 & {\GrTeXBox{v}}\arx1[dl]\arx1[dr]\\
 {\GrTeXBox{T_1}} & *{} & {\GrTeXBox{\bullet}\arx1[dl]} \\
 & {\GrTeXBox{T_2}} \\
}
\end{equation}
and define $T'=L_v(T)$ as the result of exchanging $T_1$ and $T_2$:
\begin{equation}
\label{rew-petitpaq}
\xymatrix@R=0.1cm@C=1mm{
 & {\GrTeXBox{v}}\arx1[dl]\arx1[dr]\\
 {\GrTeXBox{T_1}} & *{} & {\GrTeXBox{\bullet}\arx1[dl]} \\
 & {\GrTeXBox{T_2}} \\
}
\ \ \mapsto \ \
\xymatrix@R=0.1cm@C=1mm{
 & {\GrTeXBox{v}}\arx1[dl]\arx1[dr]\\
 {\GrTeXBox{T_2}} & *{} & {\GrTeXBox{\bullet}\arx1[dl]} \\
 & {\GrTeXBox{T_1}} \\
}
\end{equation}

Define an equivalence relation by $T\equiv_L T'$ iff $T'=L_v(T)$ for some
vertex $v$. The equivalence classes will be called L-classes.

Note that the numbers of left and right branches are constant in each
equivalence class. 

For example, here are the L-classes of (incomplete) trees of sizes
$3$ and $4$.
In size $3$, there are four L-classes for five binary trees so that
only one class contains two trees:
\begin{equation}
\Btacb \text{\ \ and\ } \Btcab.
\end{equation}

In size $4$, there are ten classes. Here are the three
non-trivial ones.
\begin{equation}
\label{ex-taille4}
\left\{ \arbm,\arbl,\arbk \right\}, \left\{ \arbi,\arbh \right\},
\left\{ \arbg,\arbe \right\}.
\end{equation}

\subsubsection{Combinatorial encoding of the L-classes}

The L-classes can be parametrized by certain bicolored plane trees. 
To see this,
let $C$ be an L-class and let $T\in C$.

Number its vertices in infix order and build two set partitions $L(T)$ and
$R(T)$ where $L(T)$ (resp. $R(T)$) consists of the sets of labels of nodes
along left (resp. right) branches.

The \emph{left} subsets will correspond to white nodes and the right ones to
black nodes.

Now build a bipartite plane tree, whose root is representing the
sequence of $R(T)$ containing the root of $T$ and such that all nodes below a
given node are the subsets of the other color (except the one encoding its own
parent) that have an intersection with the given node in increasing order of
the intersection.

For example, let us consider the trees of the L-class \eqref{ex-taille4}.
Their labelings are
\begin{equation}
\xymatrix@R=0.1cm@C=1mm{
 & {\GrTeXBox{2}}\arx1[dl]\arx1[dr]\\
 {1} & *{} & {\GrTeXBox{3}}\arx1[dr]\\
 *{} & *{} & *{} & {4} \\
}
\text{\ ,\ }
\xymatrix@R=0.1cm@C=1mm{
 & {\GrTeXBox{1}}\arx1[dr]\\
 *{} & *{} & {{\GrTeXBox{3}}\arx1[dl]}\arx1[dr] \\
 *{} & {2} & *{} & {4} \\
}
\text{\ and\ }
\xymatrix@R=0.1cm@C=1mm{
 & {\GrTeXBox{1}}\arx1[dr]\\
 *{} & *{} & {{\GrTeXBox{2}}\arx1[dr]} \\
 *{} & *{} & *{} & {{\GrTeXBox{4}}\arx1[dl]} \\
 *{} & *{} & {3}\\
}
\end{equation}
so that their corresponding bipartite trees are
\begin{equation}
\label{bipart-taille4}
\xymatrix@R=0.1cm@C=1mm{
 & {\GrTeXBox{\bullet}}\arx1[dl]\arx1[d]\arx1[dr]\\
 {\GrTeXBox{\circ}}\arx1[d] & {\GrTeXBox{\circ}} & {\GrTeXBox{\circ}}\\
 {\GrTeXBox{\bullet}} & *{} & *{} & *{} \\
}
\text{\ ,\ }
\xymatrix@R=0.1cm@C=1mm{
 & {\GrTeXBox{\bullet}}\arx1[dl]\arx1[d]\arx1[dr]\\
 {\GrTeXBox{\circ}} & {\GrTeXBox{\circ}}\arx1[d] & {\GrTeXBox{\circ}}\\
 *{} & {\GrTeXBox{\bullet}} & *{} & *{} \\
}
\text{\ and\ }
\xymatrix@R=0.1cm@C=1mm{
 & {\GrTeXBox{\bullet}}\arx1[dl]\arx1[d]\arx1[dr]\\
 {\GrTeXBox{\circ}} & {\GrTeXBox{\circ}} & {\GrTeXBox{\circ}}\arx1[d]\\
 *{} & *{} & {\GrTeXBox{\bullet}} \\
}
\end{equation}
As a more substantial example, consider the labeled  tree
\begin{equation}
\xymatrix@R=0.1cm@C=1mm{
 *{} & *{} & {\GrTeXBox{5}}\arx1[dl]\arx1[dr]\\
 *{} & {\GrTeXBox{2}}\arx1[dl]\arx1[dr] & *{} & {\GrTeXBox{6}}\arx1[dr] \\
 {\GrTeXBox{1}} & *{} & {\GrTeXBox{4}}\arx1[dl] & *{} &
    {\GrTeXBox{7}} \\
*{} & {3} \\
}
\end{equation}
for which
$L(T)=\{125,34,6,7\}$ and $R(T)=\{567,24,1,3\}$ so that the corresponding
bipartite tree is
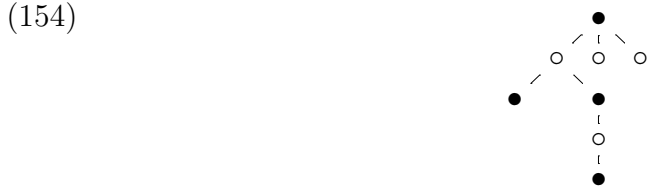
\begin{equation}
\xymatrix@R=0.1cm@C=1mm{
 *{} & *{} & {\GrTeXBox{\bullet}}\arx1[dl]\arx1[d]\arx1[dr]\\
 *{} & {\GrTeXBox{\circ}}\arx1[dl]\arx1[dr] & {\GrTeXBox{\circ}}
     & {\GrTeXBox{\circ}} \\
 {\GrTeXBox{\bullet}} & *{} & {\GrTeXBox{\bullet}}\arx1[d] \\
 *{} & *{} & {\GrTeXBox{\circ}}\arx1[d] \\
 *{} & *{} & {\GrTeXBox{\bullet}} \\
}
\end{equation}

Note that in that case, the resulting bipartite tree is not the image of the
initial binary tree by the classical bijection between plane trees and binary
trees (Knuth rotation).
However, this operation is a bijection since the number of children of the
(black) root gives the length of the right branch starting from the root and
each subtree corresponds itself recursively to a binary tree.

Actually, this correspondence coincides with the twisted Knuth rotation
defined in \cite{BS}.

\begin{theorem}
The L-classes are indexed by bipartite trees with black roots up to reordering
of the children of the \emph{black} nodes.
\end{theorem}

\Proof
The exchange rule between two binary trees along a right edge amounts to the
exchange of the corresponding subtrees of the black node corresponding to the
right branch containing this right edge. 
\qed

We shall need at some point a similar encoding, now with a white root. The
process is the same except that the root is the part of $L(T)$ that contains
the label of the root of $T$. These trees can be obtained from the
black-rooted trees by moving above the root its first (white-rooted) subtree.
Again with the same three bipartite trees, we get
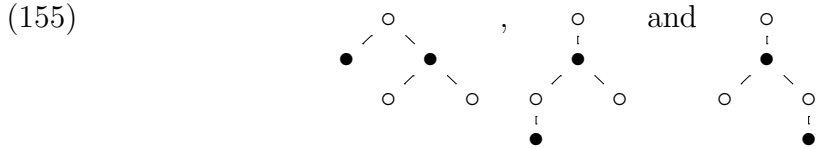
\begin{equation}
\xymatrix@R=0.1cm@C=1mm{
 & {\GrTeXBox{\circ}}\arx1[dl]\arx1[dr]\\
 {\GrTeXBox{\bullet}} && {\GrTeXBox{\bullet}}\arx1[dl]\arx1[dr] \\
 & {\GrTeXBox{\circ}} && {\GrTeXBox{\circ}} \\
}
\text{\ ,\ }
\xymatrix@R=0.1cm@C=1mm{
 & {\GrTeXBox{\circ}}\arx1[d]\\
 & {\GrTeXBox{\bullet}}\arx1[dl]\arx1[dr] \\
 {\GrTeXBox{\circ}}\arx1[d] && {\GrTeXBox{\circ}} \\
 {\GrTeXBox{\bullet}} \\
}
\text{\ and\ }
\xymatrix@R=0.1cm@C=1mm{
 & {\GrTeXBox{\circ}}\arx1[d]\\
 & {\GrTeXBox{\bullet}}\arx1[dl]\arx1[dr] \\
 {\GrTeXBox{\circ}} && {\GrTeXBox{\circ}}\arx1[d] \\
 && {\GrTeXBox{\bullet}} \\
}
\end{equation}

Note that in the case of the white-rooted bipartite trees we find in general
several classes of bipartite trees up to reordering of the children of the
black nodes.

\subsubsection{The \'Ecalle-Chapoton conjecture}

The L-classes arose while investigating the following question: given a
generic linear combination $C$ of the $T(\sigma)$, regroup the $T(\sigma)$
that always have same coefficient if one requires that $C$ belong to $\PBT$.
They conjecturally are the $L$-classes.

\begin{conjecture}
Any Lie element of $\PBT$ is a linear combination of L-classes.
\end{conjecture}

The number of L-classes in size $n$ is greater that the size of $\PBT_n\cap
\Lie(n)$ so that the conjecture provides only a necessary condition.

Here are some tables to clarify this point.
Define $a_{n,k}$ as the dimension of $\Lie(n)\cap \PBT_{n,k}$, where
$\PBT_{n,k}$ is the subspace of $\PBT_n$ spanned by the binary trees with $k$
right branches.

The $a_{n,k}$ are conjectured to be Sequence $A055277$, which is the
number of rooted trees with $n$ nodes (hence the dimension of the free preLie
algebra on one generator as conjectured by \'Ecalle and Chapoton) refined by
their number of leaves (parameter $k$).

\begin{figure}[ht]
\begin{equation}
\begin{array}{|c||c|c|c|c|c|c|c|}
\hline
n\backslash k
    & 0 & 1 & 2 &  3  &  4  &  5  & 6 \\
\hline
2   & 1 &   &   &     &     &     &   \\
\hline
3   & 1 & 1 &   &     &     &     &   \\
\hline
4   & 1 & 2 &  1&     &     &     &   \\
\hline
5   & 1 & 4 &  3&  1  &     &     &   \\
\hline
6   & 1 & 6 &  8&  4  &  1  &     &   \\
\hline
7   & 1 & 9 & 18& 14  &  5  &  1  &   \\
\hline
8   & 1 & 12& 35& 39  & 21  &  6  & 1 \\
\hline
\end{array}
\hskip.5cm
\begin{array}{c}
~ \\
A055277 \\  
\\  %
\\
\\
\end{array}
\end{equation}
\caption{\label{small-ank} The $a_{n,k}$.}
\end{figure}

Let now $b_{n,k}$ be the number of L-classes with $n-1$ nodes
and $k$ right branches.

\begin{figure}[ht]
\begin{equation}
\begin{array}{|c||c|c|c|c|c|c|c|}
\hline
n\backslash k
    & 0 & 1 & 2 &  3  &  4  &  5  & 6 \\
\hline
2   & 1 &   &   &     &     &     &   \\
\hline
3   & 1 & 1 &   &     &     &     &   \\
\hline
4   & 1 & 2 &  1&     &     &     &   \\
\hline
5   & 1 & 5 &  3&  1  &     &     &   \\
\hline
6   & 1 & 8 & 11&  4  &  1  &     &   \\
\hline
7   & 1 & 13& 30& 20  &  5  &  1  &   \\
\hline
8   & 1 & 18& 67& 73  & 31  &  6  & 1 \\
\hline
\end{array}
\end{equation}
\caption{\label{small-bnk} The $b_{n,k}$.}
\end{figure}

The first discrepance between $a_{n,k}$ and $b_{n,k}$ occurs at $n=5$.
Indeed, all five L-classes do not belong to $\PBT$.
The five L-classes with one right branch are (represented as incomplete binary
trees)
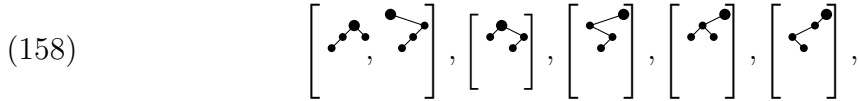
\begin{equation}
\left[ \arbg,\arbe \right],
\left[ \arbj \right],
\left[ \arbc \right],
\left[ \arbd \right],
\left[ \arbb \right],
\end{equation}

Let us denote these L-classes by $C_1,\dots,C_5$.
Expanding these as combinations of permutations in $\FQSym_5$, one finds that
only $C_5$ belongs to $\PBT_5$.
The linear span of the other ones has a 3-dimensional intersection with
$\PBT_5$. A linear basis of this intersection is given, \emph{e.g.}, by
\begin{equation}
C_1+C_4 \quad,\quad C_2-C_4 \quad,\quad C_3+C_4.
\end{equation}

Even if many properties of L-classes are conjectural, the following
result is already of interest.

\begin{theorem}
Let $\SC(n)$ be the linear subspace of $\Lie(n)$ (and therefore of the Catalan
subalgebra ${\CC}$)
spanned by the L-classes.
Then $\SC(n)$ is a sub pre-Lie algebra of ${\CC}$.\\

The pre-Lie product $s_1\droit s_2$ of two L-classes can be computed
as follows: 

let $B_1$ and $B_2$ be their respective representatives as
black-rooted bipartite trees and let $B'_1$ be the non-equivalent
representatives as white-rooted bipartite trees of $s_1$.
Then the product is obtained as the following sum:
\begin{itemize}
\item
with a minus sign, all trees obtained by gluing $B_1$ as a child of a white
node of $B_2$;
\item
with a plus sign, all trees obtained by gluing any element of $B'_1$ as a
child of a black node $B_2$.
\end{itemize}
The multiplicity of a given tree is equal to the number of ways to cut it into
two parts such that $s_2$ is the part containing the root of the tree and
$s_1$ is the other part.
\end{theorem}

\Proof
Recall that the preLie product of two naked trees $T'\droit T''$
is obtained by gluing in all possible ways $T'$ on the middle of all
branches of $T''$ (adding an invisible root to $T''$ so that it is its right
subtree), the sign of the result depending on the branch being left
(-1) or right (+1), see 
\eqref{eq-prelie-nus}.

Consider two L-classes $C'$ and $C''$ and let
$T_1$ be a tree occuring in $C'\droit C''$.
Let now $T_2$ be a tree obtained by a right branch exchange from $T_1$ (hence
in the same L-class as $T_1$). We want to show that the coefficient
of $T_2$ in the preLie product $C'\droit C''$ is the same as the coefficient
of $T_1$ in the same product.

To do this, let us encode all the ways of obtaining $T_1$ in $T_1'\droit T_1''$
with $T'_1\in C'$ and $T''_1\in C''$ as a set of triples
$(T'_1,T''_1,a)$ where $a\in T''_1$ is the branch 
used to insert
$T'_1$ inside $T''_1$.

Now, follow the branch $a$ through the right branch exchange sending $T_1$ to
$T_2$ and contract it (that is, reverse the gluing process on it).
We then get two trees, $T'_2$ and $T''_2$. Now, several cases arise.
If $a$ was not itself the right branch on which the exchange occurred, it is
obvious that: either $T'_2=T'_1$ and $T''_2$ is obtained from $T''_1$ by an
exchange along a right branch, or $T''_2=T''_1$ and $T'_2$ is obtained from
$T'_1$ by an exchange along a right branch.
Moreover, if $a$ was the exchange branch, then the parent branch of $a$ was
itself a right branch (condition to be allowed to exchange subtrees) and $T_2$
is then obtained by gluing $T'_1$ onto the father branch of $a$ in $T''_1$.

So each triple of $T_1$ is sent to a different triple of $T_2$. Since their
roles can be reversed, this injection is in fact a bijection, which proves
the required result.
\qed

Let us now describe the pre-Lie product of two L-classes on their
natural combinatorial encodings.

For example, let $s_1$ be the L-class of
$\arbz$
and $s_2$ be the L-class~\eqref{ex-taille4}.
Then $B_1$ and $B_2$ are respectively
\begin{equation}
\xymatrix@R=0.1cm@C=1mm{
 && {\GrTeXBox{\bullet}}\arx1[d]\\
 &&{\GrTeXBox{\circ}}\arx1[d] \\
 &&{\GrTeXBox{\bullet}}\\
}
\text{\qquad and\qquad}
\xymatrix@R=0.1cm@C=1mm{
 & {\GrTeXBox{\bullet}}\arx1[dl]\arx1[d]\arx1[dr]\\
 {\GrTeXBox{\circ}}\arx1[d] & {\GrTeXBox{\circ}} & {\GrTeXBox{\circ}}\\
 {\GrTeXBox{\bullet}} & *{} & *{} & *{} \\
}
\end{equation}
and the product $s_1\droit s_2$ is
\begin{equation}
- \xymatrix@R=0.1cm@C=1mm{
 & {\GrTeXBox{\bullet}}\arx1[dl]\arx1[d]\arx1[dr]\\
 {\GrTeXBox{\circ}}\arx1[d] & {\GrTeXBox{\circ}} & {\GrTeXBox{\circ}}\arx1[d]\\
 {\GrTeXBox{\bullet}} & *{} & {\GrTeXBox{\bullet}}\arx1[d] \\
 *{} & *{} & {\GrTeXBox{\circ}}\arx1[d] \\
 *{} & *{} & {\GrTeXBox{\bullet}} \\
}
- \xymatrix@R=0.1cm@C=1mm{
 & {\GrTeXBox{\bullet}}\arx1[dl]\arx1[d]\arx1[dr]\\
 {\GrTeXBox{\circ}}\arx1[d]\arx1[dr] & {\GrTeXBox{\circ}} & {\GrTeXBox{\circ}} \\
 {\GrTeXBox{\bullet}}\arx1[d] & {\GrTeXBox{\bullet}} \\
 {\GrTeXBox{\circ}}\arx1[d] \\
 {\GrTeXBox{\bullet}} \\
}
- \xymatrix@R=0.1cm@C=1mm{
 & {\GrTeXBox{\bullet}}\arx1[dl]\arx1[d]\arx1[dr]\\
 {\GrTeXBox{\circ}}\arx1[d]\arx1[dr] & {\GrTeXBox{\circ}} & {\GrTeXBox{\circ}} \\
 {\GrTeXBox{\bullet}} & {\GrTeXBox{\bullet}}\arx1[d] \\
 & {\GrTeXBox{\circ}}\arx1[d] \\
 & {\GrTeXBox{\bullet}} \\
}
+ \xymatrix@R=0.1cm@C=1mm{
 & {\GrTeXBox{\bullet}}\arx1[dl]\arx1[d]\arx1[dr]\arx1[drr]\\
 {\GrTeXBox{\circ}}\arx1[d] & {\GrTeXBox{\circ}} & {\GrTeXBox{\circ}} &
 {\GrTeXBox{\circ}}\arx1[dl]\arx1[d] \\
 {\GrTeXBox{\bullet}} && {\GrTeXBox{\bullet}} & {\GrTeXBox{\bullet}} \\
}
+ \xymatrix@R=0.1cm@C=1mm{
 & {\GrTeXBox{\bullet}}\arx1[dl]\arx1[d]\arx1[dr]\\
 {\GrTeXBox{\circ}}\arx1[d] & {\GrTeXBox{\circ}} & {\GrTeXBox{\circ}} \\
 {\GrTeXBox{\bullet}}\arx1[d] \\
 {\GrTeXBox{\circ}}\arx1[d]\arx1[dr] \\
 {\GrTeXBox{\bullet}} & {\GrTeXBox{\bullet}} \\
}
\end{equation}
whereas the product $s_2 \droit s_1$ is
\begin{equation}
- \xymatrix@R=0.1cm@C=1mm{
 & {\GrTeXBox{\bullet}}\arx1[d] \\
 &{\GrTeXBox{\circ}}\arx1[dl]\arx1[d] \\
 {\GrTeXBox{\bullet}} & {\GrTeXBox{\bullet}}\arx1[dl]\arx1[d]\arx1[dr] \\
 {\GrTeXBox{\circ}}\arx1[d] & {\GrTeXBox{\circ}} & {\GrTeXBox{\circ}}\\
 {\GrTeXBox{\bullet}} \\
}
- \xymatrix@R=0.1cm@C=1mm{
 & {\GrTeXBox{\bullet}}\arx1[d] \\
 & {\GrTeXBox{\circ}}\arx1[d]\arx1[dr] \\
 & {\GrTeXBox{\bullet}}\arx1[dl]\arx1[d]\arx1[dr] & {\GrTeXBox{\bullet}}
\\
 {\GrTeXBox{\circ}}\arx1[d] & {\GrTeXBox{\circ}} & {\GrTeXBox{\circ}}\\
 {\GrTeXBox{\bullet}} \\
}
+ \xymatrix@R=0.1cm@C=1mm{
 & {\GrTeXBox{\bullet}}\arx1[d] \\
 &{\GrTeXBox{\circ}}\arx1[d] \\
 & {\GrTeXBox{\bullet}}\arx1[d] \\
 & {\GrTeXBox{\circ}}\arx1[d]\arx1[dr] \\
 & {\GrTeXBox{\bullet}}\arx1[d]\arx1[dr] & {\GrTeXBox{\bullet}}\\
 & {\GrTeXBox{\circ}} & {\GrTeXBox{\circ}} \\
}
+ \xymatrix@R=0.1cm@C=1mm{
 & {\GrTeXBox{\bullet}}\arx1[d]\arx1[dr] \\
 & {\GrTeXBox{\circ}}\arx1[d]\arx1[dr] & {\GrTeXBox{\circ}}\arx1[dr] \\
 & {\GrTeXBox{\bullet}}\arx1[d]\arx1[dr] & {\GrTeXBox{\bullet}} &
   {\GrTeXBox{\bullet}} \\
 & {\GrTeXBox{\circ}} & {\GrTeXBox{\circ}} \\
}
+ \xymatrix@R=0.1cm@C=1mm{
 & {\GrTeXBox{\bullet}}\arx1[d]\arx1[dr] \\
 & {\GrTeXBox{\circ}}\arx1[d] & {\GrTeXBox{\circ}}\arx1[d] \\
 & {\GrTeXBox{\bullet}}\arx1[d]\arx1[dr] & {\GrTeXBox{\bullet}} \\
 & {\GrTeXBox{\circ}}\arx1[d] & {\GrTeXBox{\circ}} \\
 & {\GrTeXBox{\bullet}} \\
}
+ \xymatrix@R=0.1cm@C=1mm{
 & {\GrTeXBox{\bullet}}\arx1[d] \\
 &{\GrTeXBox{\circ}}\arx1[d] \\
 & {\GrTeXBox{\bullet}}\arx1[d] \\
 & {\GrTeXBox{\circ}}\arx1[d] \\
 & {\GrTeXBox{\bullet}}\arx1[d]\arx1[dr] \\
 & {\GrTeXBox{\circ}}\arx1[d] & {\GrTeXBox{\circ}} \\
 & {\GrTeXBox{\bullet}}
}
\end{equation}

\subsection{LR-classes of binary trees}

\subsubsection{R-equivalence}

Exchanging left and right, we have a symmetrical  notion of R-equivalence.
For a binary tree $T$, choose a vertex $v$
that can either be the root of $T$ or the left child of its parent. Then
let $T_1$ and $T_2$ be as depicted below
\begin{equation}
\xymatrix@R=0.1cm@C=1mm{
 & {\GrTeXBox{v}}\arx1[dl]\arx1[dr]\\
 {\GrTeXBox{\bullet}}\arx1[dr] & *{} & {\GrTeXBox{T_1}} \\
 & {\GrTeXBox{T_2}} \\
}
\end{equation}
and define $T'=R_v(T)$ as the result of exchanging $T_1$ and $T_2$:
\begin{equation}
\xymatrix@R=0.1cm@C=1mm{
 & {\GrTeXBox{v}}\arx1[dl]\arx1[dr]\\
 {\GrTeXBox{\bullet}}\arx1[dr] & *{} & {\GrTeXBox{T_1}} \\
 & {\GrTeXBox{T_2}} \\
}
\ \ 
\mapsto
\ \
\xymatrix@R=0.1cm@C=1mm{
 & {\GrTeXBox{v}}\arx1[dl]\arx1[dr]\\
 {\GrTeXBox{\bullet}}\arx1[dr] & *{} & {\GrTeXBox{T_2}} \\
 & {\GrTeXBox{T_1}} \\
}
\end{equation}

The union of this relation and its twin relation 
\eqref{rew-petitpaq} defines the LR-equivalence:
$T'\equiv_{LR}T$ iff $T'=L_v(T)$ or $T'=R_v(T)$ for some vertex $v$.

Again, the numbers of left and right branches are constant in
equivalence classes. 

For example, here are all LR-classes of (incomplete) trees of sizes
$3$ and $4$.
In size $3$, there are three LR-classes for five binary trees and
one class contains three trees:
\begin{equation}
\left\{ \Btacb, \Btcab, \Btbac \right\}
\end{equation}

On size $4$, there are six LR-classes. Here are the four
non-trivial ones.
\begin{equation}
\label{large-taille4}
\left\{ \arbm, \arbl, \arbk, \arbf \right\},
\left\{ \arbi, \arbh \right\},
\left\{ \arbg, \arbe, \arbd, \arbb \right\},
\left\{ \arbj, \arbc \right\}.
\end{equation}

\subsubsection{Combinatorial encoding of the LR-classes}

Consider an LR-class $C$ and a tree $T$ belonging to it.
Number its nodes as above in infix order and build two set partitions $L(T)$
and $R(T)$ where $L(T)$ (resp. $R(T)$) consists of the subsets of numbers of
nodes along left (resp. right) branches.
The \emph{left} subsets will correspond to white nodes and the right ones to
black nodes.

Now build a bipartite  free tree (\emph{i.e.}, neither rooted nor ordered)
whose nodes represent the blocks of  $L(T)$ or $R(T)$ (with the appropriate
colors) such that two nodes are connected iff their intersection in nonempty.

For example, consider the trees of the LR-class \eqref{large-taille4}. 
The first three give rise to the same
tree (see Equations~\eqref{ex-taille4} and~\eqref{bipart-taille4})
and the last one is labelled as
\begin{equation}
\xymatrix@R=0.1cm@C=1mm{
 & {\GrTeXBox{4}}\arx1[dl]\\
 {\GrTeXBox{1}}\arx1[dr]\\
 *{} & {\GrTeXBox{2}}\arx1[dr]\\
 *{} & *{} & 3 \\
}
\end{equation}
so that its (plane bipartite) tree is
\begin{equation}
\xymatrix@R=0.1cm@C=1mm{
 & {\GrTeXBox{\bullet}}\arx1[d]\\
 & {\GrTeXBox{\circ}}\arx1[d]\\
 *{} & {\GrTeXBox{\bullet}}\arx1[d]\arx1[dr]\\
 *{} & {\GrTeXBox{\circ}} & {\GrTeXBox{\circ}} \\
}
\end{equation}
which is again topologically equivalent to the three previous trees.

\begin{theorem}
The LR-classes are parametrized by bipartite  free trees.
\end{theorem}

\Proof
LR-classes are built by successively applying either an exchange of subtrees
along a left or a right branch. None of these change the tree associated
with their plane trees (defined as the encoding of their associated L-classes). 
Indeed, a right exchange corresponds to an exchange of children of a black
node whereas a left exchange corresponds to an exchange of children of a white
node. 
The statement then follows since this is an equivalence.
\qed

\subsection{LR-classes and the descent algebra}

The LR-classes arose while investigating the following question: given a
generic linear combination $g$ of the $\c_t$, regroup the $\c_t$ that
must have the same coefficient if one requires that $C$ belongs to $\NCSF$.

In this case, one can prove the following result:
\begin{theorem}
The intersection $\Lie\cap\NCSF$ is contained in the linear span of the
LR-classes.
\end{theorem}

\Proof It is known that $\Lie$ is (strictly) contained in the primitive Lie
algebra of $\FQSym$ ({\it cf.} \cite{PaR}; this is also an easy consequence of
Ree's criterion).
The coproduct of $\NCSF$ is the restriction of that of $\FQSym$, so all
elements of  $\Lie\cap\NCSF$ are primitive. Conversely, the primitive Lie
algebra of $\NCSF$ is generated by the Dynkin elements $\Psi_n$, which is an
element of the PBW basis, and is alone in its LR-class (its bipartite free
tree is the star with a black vertex at the center).
Hence, the result follows from Proposition \ref{LRLie} below.
\qed

The number of LR-classes in size $n$ is greater that the dimension of
$\NCSF_n\cap \Lie(n)$, so that an LR-class is not necessarily in $\NCSF$.

Here are some tables to clarify this point.

Let us define $a'_{n,k}$ as the dimension of $\Lie(n)$ intersected with
$\NCSF_{n,k}$ where $\NCSF_{n,k}$ is the subspace of $\NCSF$ spanned by all
ribbons with $k$ parts.

Note that the $a'_{n,k}$ are known (from external considerations) to be
Sequence $A055277$, which is the number of Lyndon words in two letters $a$ and
$b$ with $k$ times letter $b$. 

\begin{figure}[ht]
\begin{equation}
\begin{array}{|c||c|c|c|c|c|c|c|}
\hline
n\backslash k
    & 0 & 1 & 2 &  3  &  4  &  5  & 6 \\
\hline
2   & 1 &   &   &     &     &     &   \\
\hline
3   & 1 & 1 &   &     &     &     &   \\
\hline
4   & 1 & 1 & 1 &     &     &     &   \\
\hline
5   & 1 & 2 & 2 &  1  &     &     &   \\
\hline
6   & 1 & 2 & 3 &  2  &  1  &     &   \\
\hline
7   & 1 & 3 & 5 &  5  &  3  &  1  &   \\
\hline
8   & 1 & 3 & 7 &  8  &  7  &  3  & 1 \\
\hline
\end{array}
\hskip.5cm
\begin{array}{c}
~ \\
A245558 \\  
\\  %
\\
\\
\end{array}
\end{equation}
\caption{\label{large-ank} The $a'_{n,k}$.}
\end{figure}

Let now $b'_{n,k}$ be the number of LR-classes with $n-1$ nodes
and $k$ right branches.

\begin{figure}[ht]
\begin{equation}
\begin{array}{|c||c|c|c|c|c|c|c|}
\hline
n\backslash k
    & 0 & 1 & 2 &  3  &  4  &  5  & 6 \\
\hline
2   & 1 &   &   &     &     &     &   \\
\hline
3   & 1 & 1 &   &     &     &     &   \\
\hline
4   & 1 & 1 & 1 &     &     &     &   \\
\hline
5   & 1 & 2 & 2 &  1  &     &     &   \\
\hline
6   & 1 & 2 & 4 &  2  &  1  &     &   \\
\hline
7   & 1 & 3 & 7 &  7  &  3  &  1  &   \\
\hline
8   & 1 & 3 & 10& 14  & 10  &  3  & 1 \\
\hline
\end{array}
\hskip.5cm
\begin{array}{c}
~ \\
A122085 \\  
\\  %
\\
\\
\end{array}
\end{equation}
\caption{\label{large-bnk} The $b'_{n,k}$.}
\end{figure}

The first discrepance between $a'_{n,k}$ and $b'_{n,k}$ occurs at $n=6$.
Indeed, all ten LR-classes do not belong to $\NCSF$.
The discrepance occurs with $k=2$, hence on free trees with $3$ black and
$3$ white nodes.
They are
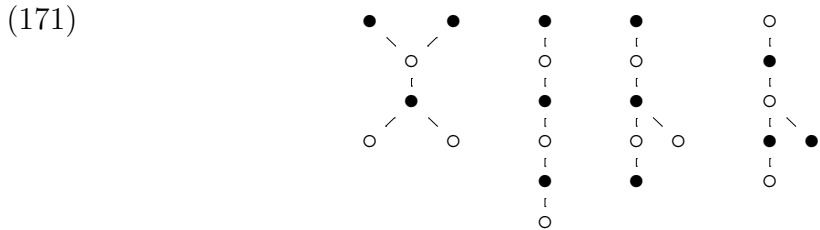
\begin{equation}
\xymatrix@R=0.1cm@C=1mm{
 & {\GrTeXBox{\bullet}}\arx1[dr] & *{} & {\GrTeXBox{\bullet}}\arx1[dl] \\
 && {\GrTeXBox{\circ}}\arx1[d]\\
 & *{} & {\GrTeXBox{\bullet}}\arx1[dl]\arx1[dr]\\
 & {\GrTeXBox{\circ}} & *{} & {\GrTeXBox{\circ}} \\
}
\quad
\xymatrix@R=0.1cm@C=1mm{
 & {\GrTeXBox{\bullet}}\arx1[d] \\
 & {\GrTeXBox{\circ}}\arx1[d]  \\
 & {\GrTeXBox{\bullet}}\arx1[d] \\
 & {\GrTeXBox{\circ}}\arx1[d]  \\
 & {\GrTeXBox{\bullet}}\arx1[d] \\
 & {\GrTeXBox{\circ}}  \\
}
\quad
\xymatrix@R=0.1cm@C=1mm{
 & {\GrTeXBox{\bullet}}\arx1[d] \\
 & {\GrTeXBox{\circ}}\arx1[d]\\
 & {\GrTeXBox{\bullet}}\arx1[d]\arx1[dr]\\
 & {\GrTeXBox{\circ}}\arx1[d] & {\GrTeXBox{\circ}} \\
 & {\GrTeXBox{\bullet}} \\
}
\quad
\xymatrix@R=0.1cm@C=1mm{
 & {\GrTeXBox{\circ}}\arx1[d] \\
 & {\GrTeXBox{\bullet}}\arx1[d]\\
 & {\GrTeXBox{\circ}}\arx1[d]\arx1[dr]\\
 & {\GrTeXBox{\bullet}}\arx1[d] & {\GrTeXBox{\bullet}} \\
 & {\GrTeXBox{\circ}} \\
}
\end{equation}

Let us denote these LR-classes by $L_1,\dots,L_4$.
Expanding these as combinations of permutations in $\FQSym_6$, one finds that
none belongs to $\NCSF_6$.
However, the intersection of their linear span with $\NCSF_6$ has dimension 3,
and a linear basis is given, \emph{e.g.}, by
\begin{equation}
L_1+L_4 \quad,\quad L_2+L_4 \quad,\quad L_3-L_4.
\end{equation}

\begin{proposition}\label{LRLie}
Let $\LC(n)$ be the (vector) subspace of $\Lie(n)$ (and of naked trees)
generated by LR-classes.
Then $\LC(n)$ is a Lie algebra.
\end{proposition}

\Proof
This follows directly from the fact that the subspace generated by L-classes
is a preLie algebra with a simple product rule.
\qed

Again, directly from the product rule of the L-classes expressed as bipartite
black-rooted trees, the Lie product of the LR-classes is quite simple:

\begin{theorem}
The Lie product $[s_1,s_2]$ of two LR-classes can be computed
as follows: let $B_1$ and $B_2$ be their respective representatives as
free bipartite trees.
Then the product is obtained as the following sum:
all trees obtained by gluing $B_1$ as a child of a white (resp. black)
node of $B_2$ with a minus (resp. plus) sign.
The multiplicity of a given tree is equal to the number of ways to cut it into
two parts such that one is $s_1$ and the other one $s_2$.
\end{theorem}

\section{Appendix}

\subsection{Miscellaneous remarks}

The LR-classes admit several alternative definitions. The original one
involves elementary moves on binary trees. Another one relies on two
non-crossing partitions which can be read on a complete binary tree with $n$
leaves: label the sectors by $1,\ldots,n -1$ from left to right.  The blocks
of the first partition $\pi$ consist of the sectors which are separated by a
left branch. The blocks of the second one $\pi'$ consist of the sectors which
are separated by a right branch.
Then, $\pi$ is the Kreweras complement of $\pi'$. Define a directed graph
whose vertices are blocks of $b'$ of $\pi'$ and blocks $b$ of $\pi$. The
oriented branches are the pairs $(b',b)$ such that $b'\cap b\not=\emptyset$.
This graph can be encoded by a bicolored tree, by coloring sinks and source in
black and white, and forgetting the arrows and the labels.

\begin{conjecture}
The bracket of $\LC$ is
\begin{equation}
[x,y] = x\vdash y - y \vdash x
\end{equation}
where $\vdash$ is a Lie admissible operation, computed on bicolored trees from
a coproduct \`a la Connes-Kreimer: the coefficient of $z$ in $x\vdash y$ is
equal to the number of subgraphs of $z$ isomorphic to $x$ such that removing
its vertices yields a graph isomorphic to $y$.
\end{conjecture}

Let us also observe that the main result of \cite{BS}, the expression of the
Solomon idempotent on the basis $\c_t$ of $\CC$, is actually (as expected) an expression
of the $LR$-classes, and that moreover, the cofficient of a free tree
is the same for its two possible bicolorings.

\subsection{Admissible labellings and decreasing trees}

\begin{proposition}
\label{prop:labels}
The admissible labellings of a complete binary tree $T$ are in bijection
with the permutations whose decreasing tree has shape $t$, the binary tree
consisting of the internal nodes of $T$. Their number is therefore given by
the hook-length formula.  
\end{proposition}

\Proof
Let $T_1(\alpha)$, $T_2(\beta)$ be two admissible labellings.
Define
\begin{equation}
B(T_1(\alpha),T_2(\beta)) =
\sum_{\gf{\std(u)=\alpha,\ \std(v)=\beta}{1\in u,\ n\in v,\ \gamma=uv}}
  T(\gamma),
\end{equation}
where $T$ is the complete binary tree with left and right subtrees $T_1$ and
$T_2$.
The sum runs over a subset of the admissible labellings of $T$, and each
admissible labelling is obtained in one, and only one sum
$B(T_1(\alpha),T_2(\beta))$ for some $\alpha$ and $\beta$. Thus, the sum of
all admissible labellings of $T$ is $B_T(\gf{\bullet}{1})$, where for a
complete binary tree $T$ and a bilinear map $B$, $B_T(a)$ means the evaluation
of the same tree with $B$ at internal nodes and $a$ in all the leaves.
This yields a bijection between these labellings and the permutations of
$\SG_{n-1}$ occuring in $B'_T(1)$, where
\begin{equation}
B'(\G_\alpha,\G_\beta)
= \sum_{\gf{\std(u)=\alpha,\ \std(v)=\beta}{1\in u,\ n\in v,\ \gamma=uv}}
     \G\gamma
= \G_\alpha\succ 1\prec \G_\beta
\end{equation}
in $\FQSym$ (see \eqref{eq:gauche} and \eqref{eq:droit} below for the last
expression), which are precisely the permutations whose decreasing tree is
$t$.
\qed

For example, consider the complete binary tree $T$ 
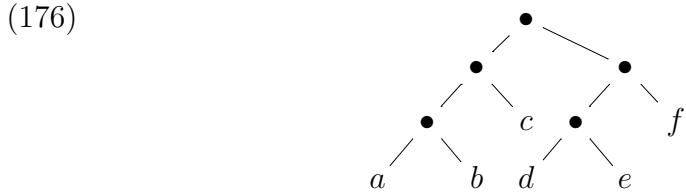
\begin{equation}
{\xymatrix@C=2mm@R=2mm{
 *{} & *{} & *{} & {\bullet}\ar@{-}[dl]\ar@{-}[drr]  \\
 *{} & *{} & {\bullet}\ar@{-}[dl]\ar@{-}[dr] & *{} & *{} &
  {\bullet}\ar@{-}[dl]\ar@{-}[dr] \\
 *{} & {\bullet}\ar@{-}[dl]\ar@{-}[dr] & *{} & c &
  {\bullet}\ar@{-}[dl]\ar@{-}[dr] & *{} & f \\
 {a} & *{} & {b} & d & *{} & e & {} & {} & {} & {}
      }}
\end{equation}
The constraints on the labellings impose that $a=1$, $f=6$, $b<c$, and $d<e$.
There are therefore 6 different labellings:
\begin{equation}
123456,\ 124356,\ 125346,\ 134256,\ 135246, 145236.
\end{equation}
These labellings are indeed in bijection with the decreasing trees of shape
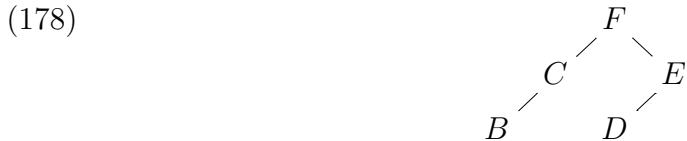
\begin{equation}
{\xymatrix@C=2mm@R=2mm{
 *{} & *{} & *{} & {F}\ar@{-}[dl]\ar@{-}[dr]  \\
 *{} & *{} & {C}\ar@{-}[dl] & *{} & {E}\ar@{-}[dl] \\
 *{} & {B} & *{} & {D} \\
      }}
\end{equation}
where the constraints fix that $F=5$, $B<C$, and $D<E$. Note that the
inequalities are exactly the same between capital and non-capital letters.

This observation allows a more direct description of the bijection. 
Start with a complete binary tree $T$ and an admissible labelling. Record into
each internal node the set of values of the leaves below it. Then remove $1$
from all subsets containing it, send each value $i$ to $i-1$ in all subsets
and then select in each subset (starting from the leaves and moving towards
the root) the smallest value that has not yet been used by the nodes below
it.

One easily checks that the resulting tree is indeed decreasing, that the
process can be reversed, and that starting with a decreasing tree, one indeed
obtains an admissible labelling of the corresponding complete binary tree.

\medskip
For example, this bijection  sends each labelling $(1,b,c,d,e,6)$ of the above
complete binary tree to the labelling $(B,C,D,E,F)=(b-1,c-1,d-1,e-1,5)$ of the
corresponding incomplete binary tree.

\subsection{The Tamari order on plane trees}

Recall the cover relation of the Tamari order on plane trees: starting from a
tree $\T$ and a vertex $x$ that is neither its root or a leaf, the trees
$\T'>\T$ covering $\T$ are obtained by cutting off the leftmost subtree of $x$
and grafting it back on the left of the parent of $x$.

The Hasse diagrams here are drawn upside down with respect to the usual
convention: the smallest element is at the top.

\def\arbta{\begin{picture}(3,4)\put(1,1){\circle*{0.7}}\put(2,2){\circle*{0.7}}\put(2,2){\Line(-1,-1)}\put(3,3){\circle*{0.7}}\put(3,3){\Line(-1,-1)}\put(3,3){\circle*{1}}\end{picture}}

\def\arbtb{\begin{picture}(3,4)\put(1,2){\circle*{0.7}}\put(2,1){\circle*{0.7}}\put(1,2){\Line(1,-1)}\put(3,3){\circle*{0.7}}\put(3,3){\Line(-2,-1)}\put(3,3){\circle*{1}}\end{picture}}

\def\arbtc{\begin{picture}(3,4)\put(1,3){\circle*{0.7}}\put(2,1){\circle*{0.7}}\put(3,2){\circle*{0.7}}\put(3,2){\Line(-1,-1)}\put(1,3){\Line(2,-1)}\put(1,3){\circle*{1}}\end{picture}}

\def\arbtd{\begin{picture}(3,3)\put(1,1){\circle*{0.7}}\put(2,2){\circle*{0.7}}\put(3,1){\circle*{0.7}}\put(2,2){\Line(-1,-1)}\put(2,2){\Line(1,-1)}\put(2,2){\circle*{1}}\end{picture}}

\def\arbte{\begin{picture}(3,4)\put(1,3){\circle*{0.7}}\put(2,2){\circle*{0.7}}\put(3,1){\circle*{0.7}}\put(2,2){\Line(1,-1)}\put(1,3){\Line(1,-1)}\put(1,3){\circle*{1}}\end{picture}}

\setlength\unitlength{1.7mm}
\newdimen\vcadre\vcadre=0.2cm 
\newdimen\hcadre\hcadre=0.2cm 

\def\cerp#1#2{\put(#1,#2){\circle*{0.7}}}
\def\cerg#1#2{\put(#1,#2){\circle*{1}}}

\def\arbtga{\begin{picture}(3,4)\cerg23\cerp{.5}1\cerp21\cerp{3.5}1
 \put(2,3){\Line(-1.5,-2)}
 \put(2,3){\Line( 0,-2)}
 \put(2,3){\Line( 1.5,-2)}
\end{picture}}

\def\arbtgb{\begin{picture}(3,6)\cerg25\cerp13\cerp33\cerp11
 \put(2,5){\Line(-1,-2)}
 \put(2,5){\Line( 1,-2)}
 \put(1,3){\Line( 0,-2)}
\end{picture}}

\def\arbtgd{\begin{picture}(3,6)\cerg25\cerp13\cerp33\cerp31
 \put(2,5){\Line(-1,-2)}
 \put(2,5){\Line( 1,-2)}
 \put(3,3){\Line( 0,-2)}
\end{picture}}

\def\arbtgc{\begin{picture}(3,6)\cerg25\cerp23\cerp11\cerp31
 \put(2,5){\Line( 0,-2)}
 \put(2,3){\Line( 1,-2)}
 \put(2,3){\Line(-1,-2)}
\end{picture}}

\def\arbtge{\begin{picture}(3,8)\cerg27\cerp25\cerp23\cerp21
 \put(2,7){\Line( 0,-2)}
 \put(2,5){\Line( 0,-2)}
 \put(2,3){\Line( 0,-2)}
\end{picture}}

\setlength\unitlength{1.7mm}
\newdimen\vcadre\vcadre=0.15cm 
\newdimen\hcadre\hcadre=0.15cm 

\begin{figure}[ht]
$\xymatrix@R=0.6cm@C=2mm{
 & *{\GrTeXBox{\arbta}}\arx1[ld]\arx1[rdd] \\
*{\GrTeXBox{\arbtb}}\arx1[d] \\
*{\GrTeXBox{\arbtc}}\arx1[rd]&
 & *{\GrTeXBox{\arbtd}}\arx1[ld] \\
 & *{\GrTeXBox{\arbte}} \\
}$
\hskip3cm
$\xymatrix@R=0.4cm@C=2mm{
 & *{\GrTeXBox{\arbtge}}\arx1[ld]\arx1[rdd] \\
*{\GrTeXBox{\arbtgc}}\arx1[d] \\
*{\GrTeXBox{\arbtgd}}\arx1[rd]&
 & *{\GrTeXBox{\arbtgb}}\arx1[ld] \\
 & *{\GrTeXBox{\arbtga}} \\
}$
\caption{\label{tam3}The Tamari order on incomplete binary trees and on the
corresponding plane trees (size $3$).}
\end{figure}

\newdimen\vcadre\vcadre=0.1cm 
\newdimen\hcadre\hcadre=0.1cm 

\def\cerp#1#2{\put(#1,#2){\circle*{0.7}}}
\def\cerg#1#2{\put(#1,#2){\circle*{1}}}

\def\arba{\begin{picture}(4,5)\cerp11\cerp22\cerp33\cerg44\put(2,2){\Line(-1,-1)}\put(3,3){\Line(-1,-1)}\put(4,4){\Line(-1,-1)}\end{picture}}

\def\arbb{\begin{picture}(4,5)\put(1,2){\circle*{0.7}}\put(2,1){\circle*{0.7}}\put(1,2){\Line(1,-1)}\put(3,3){\circle*{0.7}}\put(3,3){\Line(-2,-1)}\put(4,4){\circle*{0.7}}\put(4,4){\Line(-1,-1)}\put(4,4){\circle*{1}}\end{picture}}

\def\arbc{\begin{picture}(4,5)\put(1,3){\circle*{0.7}}\put(2,1){\circle*{0.7}}\put(3,2){\circle*{0.7}}\put(3,2){\Line(-1,-1)}\put(1,3){\Line(2,-1)}\put(4,4){\circle*{0.7}}\put(4,4){\Line(-3,-1)}\put(4,4){\circle*{1}}\end{picture}}

\def\arbd{\begin{picture}(4,5)\put(1,2){\circle*{0.7}}\put(2,3){\circle*{0.7}}\put(3,2){\circle*{0.7}}\put(2,3){\Line(-1,-1)}\put(2,3){\Line(1,-1)}\put(4,4){\circle*{0.7}}\put(4,4){\Line(-2,-1)}\put(4,4){\circle*{1}}\end{picture}}

\def\arbe{\begin{picture}(4,5)\put(1,4){\circle*{0.7}}\put(2,1){\circle*{0.7}}\put(3,2){\circle*{0.7}}\put(3,2){\Line(-1,-1)}\put(4,3){\circle*{0.7}}\put(4,3){\Line(-1,-1)}\put(1,4){\Line(3,-1)}\put(1,4){\circle*{1}}\end{picture}}

\def\arbf{\begin{picture}(4,5)\put(1,3){\circle*{0.7}}\put(2,2){\circle*{0.7}}\put(3,1){\circle*{0.7}}\put(2,2){\Line(1,-1)}\put(1,3){\Line(1,-1)}\put(4,4){\circle*{0.7}}\put(4,4){\Line(-3,-1)}\put(4,4){\circle*{1}}\end{picture}}

\def\arbg{\begin{picture}(4,4)\put(1,1){\circle*{0.7}}\put(2,2){\circle*{0.7}}\put(2,2){\Line(-1,-1)}\put(3,3){\circle*{0.7}}\put(4,2){\circle*{0.7}}\put(3,3){\Line(-1,-1)}\put(3,3){\Line(1,-1)}\put(3,3){\circle*{1}}\end{picture}}

\def\arbh{\begin{picture}(4,5)\put(1,4){\circle*{0.7}}\put(2,2){\circle*{0.7}}\put(3,1){\circle*{0.7}}\put(2,2){\Line(1,-1)}\put(4,3){\circle*{0.7}}\put(4,3){\Line(-2,-1)}\put(1,4){\Line(3,-1)}\put(1,4){\circle*{1}}\end{picture}}

\def\arbi{\begin{picture}(4,4)\put(1,2){\circle*{0.7}}\put(2,1){\circle*{0.7}}\put(1,2){\Line(1,-1)}\put(3,3){\circle*{0.7}}\put(4,2){\circle*{0.7}}\put(3,3){\Line(-2,-1)}\put(3,3){\Line(1,-1)}\put(3,3){\circle*{1}}\end{picture}}

\def\arbj{\begin{picture}(4,4)\put(1,2){\circle*{0.7}}\put(2,3){\circle*{0.7}}\put(3,1){\circle*{0.7}}\put(4,2){\circle*{0.7}}\put(4,2){\Line(-1,-1)}\put(2,3){\Line(-1,-1)}\put(2,3){\Line(2,-1)}\put(2,3){\circle*{1}}\end{picture}}

\def\arbk{\begin{picture}(4,5)\put(1,4){\circle*{0.7}}\put(2,3){\circle*{0.7}}\put(3,1){\circle*{0.7}}\put(4,2){\circle*{0.7}}\put(4,2){\Line(-1,-1)}\put(2,3){\Line(2,-1)}\put(1,4){\Line(1,-1)}\put(1,4){\circle*{1}}\end{picture}}

\def\arbl{\begin{picture}(4,5)\put(1,4){\circle*{0.7}}\put(2,2){\circle*{0.7}}\put(3,3){\circle*{0.7}}\put(4,2){\circle*{0.7}}\put(3,3){\Line(-1,-1)}\put(3,3){\Line(1,-1)}\put(1,4){\Line(2,-1)}\put(1,4){\circle*{1}}\end{picture}}

\def\arbm{\begin{picture}(4,4)\put(1,2){\circle*{0.7}}\put(2,3){\circle*{0.7}}\put(3,2){\circle*{0.7}}\put(4,1){\circle*{0.7}}\put(3,2){\Line(1,-1)}\put(2,3){\Line(-1,-1)}\put(2,3){\Line(1,-1)}\put(2,3){\circle*{1}}\end{picture}}

\def\arbn{\begin{picture}(4,5)\put(1,4){\circle*{0.7}}\put(2,3){\circle*{0.7}}\put(3,2){\circle*{0.7}}\put(4,1){\circle*{0.7}}\put(3,2){\Line(1,-1)}\put(2,3){\Line(1,-1)}\put(1,4){\Line(1,-1)}\put(1,4){\circle*{1}}\end{picture}}

\def\arbga{\begin{picture}(5,4)\cerp{.5}1\cerp21\cerp{3.5}1\cerp51
 \cerg{2.75}3
 \put(2.75,3){\Line(-2.25,-2)}
 \put(2.75,3){\Line(-.75,-2)}
 \put(2.75,3){\Line( .75,-2)}
 \put(2.75,3){\Line( 2.25,-2)}
\end{picture}}

\def\arbgb{\begin{picture}(5,5)\cerp{.5}0\cerp{.5}2\cerp22\cerp{3.5}2\cerg{2}4
 \put(2,4){\Line(-1.5,-2)}
 \put(2,4){\Line( 0,-2)}
 \put(2,4){\Line( 1.5,-2)}
 \put(.5,2){\Line( 0,-2)}
\end{picture}}

\def\arbgc{\begin{picture}(5,6)\cerp{1}1\cerp31\cerp23\cerp43\cerg35
 \put(3,5){\Line(-1,-2)}
 \put(3,5){\Line( 1,-2)}
 \put(2,3){\Line(-1,-2)}
 \put(2,3){\Line( 1,-2)}
\end{picture}}

\def\arbgd{\begin{picture}(5,5)\cerp{2}0\cerp{.5}2\cerp22\cerp{3.5}2\cerg{2}4
 \put(2,4){\Line(-1.5,-2)}
 \put(2,4){\Line( 0,-2)}
 \put(2,4){\Line( 1.5,-2)}
 \put(2,2){\Line( 0,-2)}
\end{picture}}

\def\arbgg{\begin{picture}(5,5)\cerp{3.5}0\cerp{.5}2\cerp22\cerp{3.5}2\cerg{2}4
 \put(2,4){\Line(-1.5,-2)}
 \put(2,4){\Line( 0,-2)}
 \put(2,4){\Line( 1.5,-2)}
 \put(3.5,2){\Line( 0,-2)}
\end{picture}}

\def\arbge{\begin{picture}(5,6)\cerp{1.5}1\cerp{3}1\cerp{4.5}1\cerp{3}3\cerg{3}5
 \put(3,5){\Line( 0,-2)}
 \put(3,3){\Line( 0,-2)}
 \put(3,3){\Line( -1.5,-2)}
 \put(3,3){\Line( 1.5,-2)}
\end{picture}}

\def\arbgf{\begin{picture}(4,8)\cerp{1}1\cerp{1}3\cerp15\cerp{3}5\cerg{2}7
 \put(2,7){\Line(-1,-2)}
 \put(2,7){\Line( 1,-2)}
 \put(1,5){\Line( 0,-2)}
 \put(1,3){\Line( 0,-2)}
\end{picture}}

\def\arbgh{\begin{picture}(5,8)\cerg{2}7\cerp{2}5\cerp13\cerp{3}3\cerp{1}1
 \put(2,7){\Line( 0,-2)}
 \put(2,5){\Line(-1,-2)}
 \put(2,5){\Line( 1,-2)}
 \put(1,3){\Line( 0,-2)}
\end{picture}}

\def\arbgi{\begin{picture}(5,6)\cerg{2}5\cerp{1}3\cerp33\cerp{1}1\cerp{3}1
 \put(2,5){\Line(-1,-2)}
 \put(2,5){\Line( 1,-2)}
 \put(1,3){\Line( 0,-2)}
 \put(3,3){\Line( 0,-2)}
\end{picture}}

\def\arbgj{\begin{picture}(5,6)\cerg{2}5\cerp{1}3\cerp33\cerp{2}1\cerp{4}1
 \put(2,5){\Line(-1,-2)}
 \put(2,5){\Line( 1,-2)}
 \put(3,3){\Line(-1,-2)}
 \put(3,3){\Line( 1,-2)}
\end{picture}}

\def\arbgk{\begin{picture}(4,8)\cerg{2}7\cerp{2}5\cerp23\cerp{1}1\cerp{3}1
 \put(2,7){\Line( 0,-2)}
 \put(2,5){\Line( 0,-2)}
 \put(2,3){\Line(-1,-2)}
 \put(2,3){\Line( 1,-2)}
\end{picture}}

\def\arbgl{\begin{picture}(4,8)\cerg{2}7\cerp{2}5\cerp13\cerp{3}3\cerp{3}1
 \put(2,7){\Line( 0,-2)}
 \put(2,5){\Line(-1,-2)}
 \put(2,5){\Line( 1,-2)}
 \put(3,3){\Line( 0,-2)}
\end{picture}}

\def\arbgm{\begin{picture}(4,8)\cerg{2}7\cerp{1}5\cerp35\cerp{3}3\cerp{3}1
 \put(2,7){\Line(-1,-2)}
 \put(2,7){\Line( 1,-2)}
 \put(3,5){\Line( 0,-2)}
 \put(3,3){\Line( 0,-2)}
\end{picture}}

\def\arbgn{\begin{picture}(2,10)\cerg{1}9\cerp{1}7\cerp15\cerp{1}3\cerp{1}1
 \put(1,9){\Line( 0,-2)}
 \put(1,7){\Line( 0,-2)}
 \put(1,5){\Line( 0,-2)}
 \put(1,3){\Line( 0,-2)}
\end{picture}}

\begin{figure}[ht]
\centerline{
\tiny
\rotateleft{
\hfill
\setlength\unitlength{1.7mm}
$\xymatrix@R=0.5cm@C=7mm{
&  & *{\GrTeXBox{\arba}}\arx1[ld]\arx1[dd]\arx1[rrddd]& \\
& *{\GrTeXBox{\arbb}}\arx1[ld]\arx1[rddd]& \\
*{\GrTeXBox{\arbc}}\arx1[d]\arx1[rd]
 &  & *{\GrTeXBox{\arbd}}\arx1[ld]\arx1[rdd]& \\
*{\GrTeXBox{\arbe}}\arx1[d]\arx1[rrdd]
 & *{\GrTeXBox{\arbf}}\arx1[ld]&  &  
 & *{\GrTeXBox{\arbg}}\arx1[lld]\arx1[dd]& \\
*{\GrTeXBox{\arbh}}\arx1[rd] & 
 & *{\GrTeXBox{\arbi}}\arx1[d]
 & *{\GrTeXBox{\arbj}}\arx1[lld]\arx1[rd]& \\
& *{\GrTeXBox{\arbk}}\arx1[rd]
 & *{\GrTeXBox{\arbl}}\arx1[d]&  
 & *{\GrTeXBox{\arbm}}\arx1[lld]& \\
 &  & *{\GrTeXBox{\arbn}}& \\
}$
\hfill
\setlength\unitlength{1.7mm}
$\xymatrix@R=0.5cm@C=7mm{
&  & *{\GrTeXBox{\arbgn}}\arx1[ld]\arx1[dd]\arx1[rrddd]& \\
& *{\GrTeXBox{\arbgk}}\arx1[ld]\arx1[rddd]& \\
*{\GrTeXBox{\arbgl}}\arx1[d]\arx1[rd]
 &  & *{\GrTeXBox{\arbgh}}\arx1[ld]\arx1[rdd]& \\
*{\GrTeXBox{\arbgm}}\arx1[d]\arx1[rrdd]
 & *{\GrTeXBox{\arbge}}\arx1[ld]&  &  
 & *{\GrTeXBox{\arbgf}}\arx1[lld]\arx1[dd]& \\
*{\GrTeXBox{\arbgj}}\arx1[rd] & 
 & *{\GrTeXBox{\arbgc}}\arx1[d]
 & *{\GrTeXBox{\arbgi}}\arx1[lld]\arx1[rd]& \\
& *{\GrTeXBox{\arbgg}}\arx1[rd]
 & *{\GrTeXBox{\arbgd}}\arx1[d]&  
 & *{\GrTeXBox{\arbgb}}\arx1[lld]& \\
 &  & *{\GrTeXBox{\arbga}}& \\
}$
\hfill
}}
\caption{\label{tam4}The Tamari order on incomplete binary trees and on the
corresponding plane trees (size $4$).}
\end{figure}

\def\arbgga#1#2#3#4{
\xymatrix@R=0.2cm@C=1mm{
 & {\GrTeXBox{#1}}\arx1[d] \\
 & {\GrTeXBox{#2}}\arx1[d] \\
 & {\GrTeXBox{#3}}\arx1[d] \\
 & {\GrTeXBox{#4}} \\
}
}

\def\arbggb#1#2#3#4{
\xymatrix@R=0.1cm@C=1mm{
 & {\GrTeXBox{#1}}\arx1[d] \\
 & {\GrTeXBox{#2}}\arx1[dl]\arx1[dr] \\
   {\GrTeXBox{#3}} && {\GrTeXBox{#4}} \\
}}

\def\arbggc#1#2#3#4{
\xymatrix@R=0.1cm@C=1mm{
 & {\GrTeXBox{#1}}\arx1[dl]\arx1[dr] \\
   {\GrTeXBox{#2}}\arx1[d] && {\GrTeXBox{#3}} \\
   {\GrTeXBox{#4}} \\
}}

\def\arbggd#1#2#3#4{
\xymatrix@R=0.1cm@C=1mm{
 & {\GrTeXBox{#1}}\arx1[dl]\arx1[dr] \\
   {\GrTeXBox{#2}} && {\GrTeXBox{#3}}\arx1[d] \\
 &&  {\GrTeXBox{#4}} \\
}}

\def\arbgge#1#2#3#4{
\xymatrix@R=0.1cm@C=1mm{
 & {\GrTeXBox{#1}}\arx1[dl]\arx1[d]\arx1[dr] \\
   {\GrTeXBox{#2}} & {\GrTeXBox{#3}} & {\GrTeXBox{#3}} \\
}}

\setlength\unitlength{1.7mm}
\newdimen\vcadre\vcadre=0.15cm 
\newdimen\hcadre\hcadre=0.05cm 

\begin{figure}[ht]
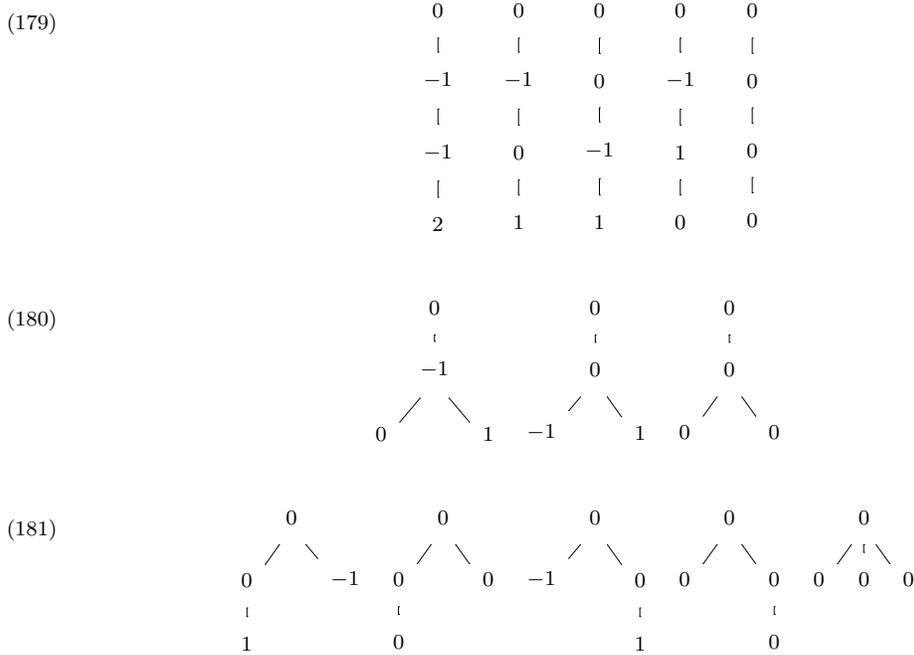

\begin{equation}
\tiny
\arbgga{0}{-1}{-1}{2}  
\arbgga{0}{-1}{0}{1} 
\arbgga{0}{0}{-1}{1}
\arbgga{0}{-1}{1}{0}
\arbgga{0}{0}{0}{0}
\end{equation}
\begin{equation}
\tiny
\arbggb{0}{-1}{0}{1}\
\arbggb{0}{0}{-1}{1}\
\arbggb{0}{0}{0}{0}
\end{equation}
\begin{equation}
\tiny
\arbggc{0}{0}{-1}{1}\
\arbggc{0}{0}{0}{0}\
\arbggd{0}{-1}{0}{1}\
\arbggd{0}{0}{0}{0}\
\arbgge{0}{0}{0}{0}
\end{equation}
\caption{\label{fig-flots}Small closed flows on plane trees with $4$ nodes}
\end{figure}



\newpage
\footnotesize
\begin{thebibliography}{aa}
%
\bibitem{AL}{\sc M. Aguiar} and {\sc  J.-L. Loday}, {\it
Quadri-algebras}, Journal of Pure and Applied Algebra {\bf  191} (2004),
205--221.
%
\bibitem{BS}{\sc R. Bandiera} and {\sc F. Schaetz}, {\it
Eulerian idempotent, pre-Lie logarithm and combinatorics of trees},
arXiv:1702.08907.
%
\bibitem{BBG} \sc F. Bergeron, N. Bergeron, \rm and \sc A.M. Garsia, \it
Idempotents for the free Lie algebra and $q$-enumeration\rm, {\it in}
Invariant theory and tableaux, D. Stanton {\it ed.}, IMA Volumes in
Mathematics and its Applications, Vol. 19, Springer, 1988.
%
\bibitem{BMP} \sc I. Bialynicki-Birula, B. Mielnik, \rm and \sc J.
Pleba\'nski, \it Explicit solution of the continuous
Baker-Campbell-Hausdorff problem\rm, Annals of Physics {\bf 51} (1969),
187-200.
%
\bibitem{Ble}{\sc D. Blessenohl},
{\it An unpublished theorem of Manfred Schocker and the Patras-Reutenauer
algebra},
J. Alg. Comb. {\bf 28} (2008), 25--42.
%
\bibitem{BL} \sc D. Blessenohl \rm and \sc H. Laue, \it Algebraic
combinatorics related to the free Lie algebra\rm, Actes du 29-i\`eme
S\'eminaire Lotharingien de Combinatoire, A. Kerber Ed., Publ. IRMA,
Strasbourg, 1993, 1-21.
%
\bibitem{Bour} \sc N. Bourbaki, \it Groupes et alg\`ebres de Lie\rm ,
Chap. 2 et 3, Hermann, 1972.

\bibitem{Cartier}{\sc P. Cartier}, {\it D\'emonstration alg\'ebrique de la formule
de Hausdorff}, Bull. Soc. Math. France {\bf 84} (1956), 241--249.
%
\bibitem{Cha1}{\sc F. Chapoton},
{\it A rooted-trees q-series lifting a one-parameter family of Lie
idempotents},
Algebra \& Number Theory,  {\bf 3} (2009), 611--636 
%
\bibitem{Cha2}{\sc F. Chapoton}, 
{\it Flows on rooted trees and the Menous-Novelli-Thibon idempotents},
Math. Scand. {\bf 115} (2014), 20--61.
%
\bibitem{CL}{\sc F. Chapoton} and {\sc M.  Livernet},
{\it Pre-Lie algebras and the rooted trees operad},
Internat. Math. Research Notices  {\bf 8} (2001), 395--408
%
\bibitem{NCSF7}{\sc G. Duchamp, F. Hivert, J.-C. Novelli}, and {\sc J.-Y.
Thibon}, {\it Noncommutative Symmetric Functions VII: Free Quasi-Symmetric
Functions Revisited},
Annals of Combinatorics {\bf 15} (2011), 655--673.
%
\bibitem{NCSF6}{\sc G. Duchamp, F. Hivert}, and {\sc J.-Y. Thibon},
{\it Noncommutative symmetric functions VI: free quasi-symmetric functions and
related algebras},
Internat. J. Alg. Comput. {\bf 12} (2002), 671--717.
%
\bibitem{NCSF3} {\sc G. Duchamp, A. Klyachko, D. Krob}, and {\sc
J.-Y. Thibon}, {\it Noncommutative symmetric functions III:
Deformations of Cauchy and convolution algebras},
Discrete Mathematics and Theoretical Computer Science {\bf 1} (1997), 159--216.
%
\bibitem{DKLT94}{\sc G. Duchamp, D. Krob, B. Leclerc}, and {\sc J.-Y. Thibon},
{\it D\'eformations de projecteurs de Lie}, C.R. Acad. Sci. Paris
{\bf 319} (1994), 909-914.
%
\bibitem{DKV}{\sc G. Duchamp, D. Krob}, and {\sc E. A. Vassilieva},
{\it Zassenhaus Lie idempotents,
$q$-bracketing and a new exponential/logarithm correspondence},  J. Algebraic
Combin.  12  (2000),  no. 3, 251--277.
%
\bibitem{Dyn1} {\sc E.B. Dynkin}, {\it Calculation of the coefficients in
the Campbell-Baker-Hausdorff formula}, Dokl. Akad. Nauk. SSSR (N.S.)
{\bf 57} (1947), 323-326 (in Russian).
%
\bibitem{Foi1}{\sc L. Foissy}, {\it
Bidendriform bialgebras, trees, and free quasi-symmetric functions},
J.Pure Appl. Algebra {\bf 209} (2007), 439--459.
%
\bibitem{Foi2}{\sc L. Foissy}, {\it
Free quadri-algebras and dual quadri-algebras},
arXiv:1504.06056.
%
\bibitem{Foi3}{\sc L. Foissy}, {\it
Free brace algebras are free prelie algebras},  Comm. Algebra {\bf 38} (2010), 3358--3369.
%

\bibitem{Fo}\sc H.O. Foulkes, \it Characters of symmetric groups induced by
char acters of cyclic subgroups\rm, in \it Combinatorics \rm (Proc. Conf. Comb. Math. Inst.
Oxford 1972), Inst. Math. Appl., Southend-on-Sea, 1972, 141-154.
%
\bibitem{Fri}{\sc K. O. Friederichs}, {\it Mathematical aspects
of the quantum theory of fields}, Interscience, New-York, 1951.
%
\bibitem{Ga} \sc A.M. Garsia, \it Combinatorics of the free Lie
algebra and the symmetric group\rm, {\it in} Analysis, et cetera ...,
J\"urgen Moser Festschrift, Academic press, New York, (1990), 309-82.
%
\bibitem{GaR} \sc A.M. Garsia {\rm and} C. Reutenauer, \it
A decomposition of Solomon's descent algebra\rm, Advances
in Math. {\bf 77} (1989), 189-262.
%
\bibitem{NCSF1}{\sc I.M. Gelfand, D. Krob, A. Lascoux, B. Leclerc,
V.~S. Retakh}, and {\sc J.-Y. Thibon},
{\it Noncommutative symmetric functions},
Adv. in Math. {\bf 112} (1995), 218--348.
%
\bibitem{HNT}{\sc F.~Hivert, J.-C.~Novelli}, and {\sc J.-Y.~Thibon}.
{\it The algebra of binary search trees},
Theoret. Comput. Sci.  {\bf 339}  (2005), 129--165. 
%
\bibitem{Kl} \sc A.A. Klyachko, \it Lie elements in the tensor
algebra\rm, Siberian Math. J. {\bf 15} (1974), 1296-1304.
%
\bibitem{Kn} {\sc D.E. Knuth}, The art of computer programming, vol.3: Sorting
and searching, Addison-Wesley, 1973.
%
\bibitem{NCSF2}{\sc D. Krob, B. Leclerc}, and {\sc J.-Y. Thibon},
{\it Noncommutative symmetric functions II: Transformations of alphabets},
Intern. J. Alg. Comput. {\bf 7} (1997), 181--264.
%
\bibitem{Lod89}{\sc J.-L. Loday},
{\it Op\'erations sur l'homologie cyclique des algèbres commutatives} 
Invent. Math. {\bf 96} (1989), 205--230.
%
\bibitem{Lod0}{\sc J.-L. Loday},
{\it Dialgebras and related operads}, in
Lecture Notes in Math. {\bf 1763}, Springer, Berlin, 2001, p. 7--66.
%
\bibitem{Lod}{\sc J.-L. Loday},
{\it Scindement d'associativit\'e et alg\`ebres de {H}opf},
Actes des Journ\'ees Math\'e\-matiques \`a la M\'emoire de Jean Leray,
S\'emin. Congr. Soc. Math. France 9 (2004), 155--172.
%
\bibitem{LR1}{\sc J.-L. Loday} and {\sc M.~O. Ronco},
{\it Hopf algebra of the planar binary trees},
Adv. Math. {\bf 139} (1998) n. 2, 293--309.
%
\bibitem{Loth}{\sc M. Lothaire}, Combinatorics on words,
Cambridge University Press (1997).
%
\bibitem{Mcd}{\sc I.G. Macdonald},
{\it Symmetric functions and Hall polynomials},
2nd ed., Oxford University Press, 1995.
%
\bibitem{Mag} \sc W. Magnus, \it On the exponential solution of
differential equations for a linear operator\rm, Comm. Pure Appl. Math.
{\bf VII} (1954), 649-673.
%
\bibitem{MR}{\sc C. Malvenuto} and {\sc C. Reutenauer},
{\it Duality between quasi-symmetric functions and the Solomon descent
algebra},
J. Algebra {\bf 177} (1995), 967--982.
%
\bibitem{Men}{\sc F. Menous}, {\it The well-behaved Catalan and Brownian
averages and their applications to real resummation},
 Proceedings of the Symposium on Planar Vector Fields 
(Lleida, 1996). Publ. Mat. {\bf 41} (1997), no. 1, 209--222.
%
\bibitem{MNT}{\sc F. Menous, J.-C. Novelli}, and {\sc J.-Y. Thibon},
{\it Mould calculus, polyhedral cones, and characters of combinatorial Hopf
algebras}, 
Advances in Applied Math. {\bf 51} (2013), 177-227.
%
\bibitem{MP} \sc B. Mielnik \rm and \sc J. Pleba\'nski, \it Combinatorial
approach to Baker-Campbell-Hausdorff exponents\rm, Ann. Inst. Henri Poincar\'e,
Section A, vol. {\bf XII} (1970), 215-254.
%
\bibitem{NT1} {\sc J.-C Novelli} and {\sc J.-Y. Thibon},
{\it Construction de trig\`ebres dendriformes},
C. R. Math. Acad. Sci, Paris {\bf 342} (2006), 365--369 (in French).
%
\bibitem{NT2} {\sc J.-C Novelli} and {\sc J.-Y. Thibon},
{\it Polynomial Realizations of some Trialgebras},
Proceedings FPSAC'06, San Diego.

\bibitem{OG}{\sc J.-M. Oudom} and {\sc D. Guin},
{\it Sur l'algèbre enveloppante d'une algèbre pré-Lie (On the Lie enveloping
algebra of a pre-Lie algebra)},
Comptes Rendus Mathematique {\bf 340} (2005),  331--336.
%
\bibitem{Pa} \sc F. Patras, \it L'alg\`ebre des descentes
d'une big\`ebre gradu\'ee\rm, J. Algebra {\bf 170} (1994), 547--566.
%
\bibitem{PaR}{\sc F. Patras} and {\sc C. Reutenauer},
{\it Lie representations and an algebra containing Solomon's}, 
 J. Alg. Comb. {\bf 16} (2002), 301--314. 
%
\bibitem{Ree} {\sc R. Ree}, {\it Generalized Lie elements}, Canad.
J.  Math. {\bf 12} (1960), 493--502.
%
\bibitem{Re86}  \sc C. Reutenauer, \it Theorem of
Poincar\'e-Birkhoff-Witt, logarithm and representations of the symmetric group
whose order are the Stirling numbers\rm, {\it in} Combinatoire \'enum\'erative,
Proceedings, Montr\'eal 1985 (G. Labelle and P. Leroux Eds.),
Lecture Notes in Math., {\bf 1234}, Springer, (1986), 267-284.
%
\bibitem{Re} {\sc C. Reutenauer}, {\it Free Lie algebras}, Oxford, 1993.
%
\bibitem{ST}{\sc P. Salvatore} and {\sc  R. Tauraso}, {\it
The operad Lie is free}, arXiv:0802.3010.
%
\bibitem{Slo} {\sc N.J.A. Sloane},
{\it The On-Line Encyclopedia of Integer Sequences},\\
\verb+http://www.research.att.com/~njas/sequences/+
%
\bibitem{So1} \sc L. Solomon, \it On the Poincar\'e-Birkhoff-Witt
theorem\rm, J. Comb. Theory {\bf 4} (1968), 363-375.
%
\bibitem{Sol} {\sc L. Solomon}, {\it A Mackey formula in the group
ring of a Coxeter group}, J.  Algebra {\bf 41} (1976), 255--268.
%
\bibitem{Spe}{\sc W. Specht}, {\it Die linearen Beziehungen zwischen
h\"oheren Kommutatoren}, Math. Zeit. {\bf 51} (1948), 367-376.
%
\bibitem{Wev} {\sc F. Wever}, {\it \"Uber Invarianten in Lieschen
Ringen}, Math. Annalen {\bf 120} (1949), 563-580.
%
\bibitem{Wil} \sc R.M. Wilcox, \it Exponential operators and
parameter differentiation in Quantum Physics\rm, J. Math. Phys.
{\bf 8} (1967), 962-982.
%
\bibitem{Witt}{\sc E. Witt},
{\it Treue Darstellung Liescher Ringe}, J. Reine Angew. Math. {\bf 177} (1937), 152--160.

\end{thebibliography}
\end{document}